\documentclass{amsart}

\usepackage{amssymb}
\usepackage{amscd}
\usepackage[all,cmtip]{xy}
\usepackage{hyperref}
\usepackage{graphicx}

\usepackage[pdftex]{color}

\newcommand{\comment}[1]{}

\newcommand{\C}{\mathbb{C}}

\newcommand{\Ps}{\mathbb{P}}
\newcommand{\Q}{\mathbb{Q}}
\newcommand{\Z}{\mathbb{Z}}
\newcommand{\R}{\mathbb{R}}
\newcommand{\F}{\mathbb{F}}

\newcommand{\pa}{\mathfrak{p}}

\DeclareMathOperator{\Gal}{Gal}

\DeclareMathOperator{\Hom}{Hom}

\DeclareMathOperator{\lcm}{lcm}

\newcommand{\mapsfrom}{\raisebox{1ex}{\rotatebox{180}{$\mapsto$}}}

\theoremstyle{plain}
\newtheorem{theorem}{Theorem}[section]

\newtheorem{corollary}[theorem]{Corollary}

\newtheorem{lemma}[theorem]{Lemma}

\newtheorem{proposition}[theorem]{Proposition}

\theoremstyle{definition}
\newtheorem{definition}[theorem]{Definition}

\newtheorem{remark}[theorem]{Remark}

\newtheorem{example}[theorem]{Example}

\begin{document}

\title[$\Z_p$-extensions]{On the arithmetic of $\Z_p$-extensions}

\author{Michiel Kosters}
\address{University of California, Irvine, Department of
Mathematics, 340 Rowland Hall, Irvine, CA 92697}
\email{kosters@gmail.com}
\author{Daqing Wan}
\address{University of California, Irvine, Department of
Mathematics, 340 Rowland Hall, Irvine, CA 92697}
\email{dwan@math.uci.edu}

\date{\today}

\maketitle

\begin{abstract}
This paper was motivated by the possible Newton slope stable property \cite{WAN7} of the zeta function in a $\Z_p$-tower 
of curves over a finite field of characteristic $p$. Our aim here is to  
develop the foundational materials for $\Z_p$-extensions of a global function field of characteristic $p$. 
This paper contains three parts. 

In the first part, we give a thorough overview of the theory of Artin-Schreier-Witt extensions: this theory allows one to understand the $\Z/p^n\Z$-extensions of any field $K$ of characteristic $p$ 
via $p$-typical Witt vectors. 
Let $W_n(K)$ be the ring of $p$-typical Witt vectors of $K$ of length $n$ and let 
$\wp = F-\mathrm{id}: W_n(K)\longrightarrow W_n(K)$, where $F$ is the Frobenius map and $\mathrm{id}$ is the identity map.  
 Artin-Schreier-Witt theory tells us that the abelian group $W_n(K)/\wp W_n(K)$ represents the set of $\Z/p^n\Z$-extensions of $K$. Since this theory is hard to find in literature, we have included 
a complete treatment in the paper.

In the second part of the paper, we study $\Z_p$-extensions of a local field $K=k((T))$ of characteristic $p>0$ where $k$ is a finite field. Local class field theory and Artin-Schreier-Witt theory give us the Schmid-Witt symbol 
\begin{align*}
[\ ,\ ): W(K)/\wp W(K) \times \widehat{K^*} \to W(\F_p)=\Z_p,
\end{align*}
which contains the ramification information of $\Z_p$-extensions of $K$. We present a new simplified formula for $[\ ,\ )$. This formula allows one to compute ramification groups, conductors and discriminants in an easy way.

In the third part, we study $\Z_p$-extensions of global function fields over a finite field. First, we give a formula for computing the genus in such a tower. We show that a previously obtained lower bound for the genus growth in a $\Z_p$-extension is incorrect and we give a sharp lower bound. We also study when the genus behaves in a `stable' way. Finally, we find unique representatives of $\Z_p$-extensions of the rational function field $k(X)$, and compute the genus in such a tower. 
\end{abstract}

\maketitle
\tableofcontents

\section{Introduction}

\subsection{$\Z_p$-extensions in characteristic $p$}

Let $K$ be a field of characteristic $p>0$. Let $l$ be a positive integer. We would like to classify the cyclic Galois extensions of $K$ of degree $l$.

Assume first that $l$ is coprime to $p$. Kummer-theory solves this problem, provided that $K$ contains a primitive $l$-th root of unity: it essentially shows that such extensions are given by polynomials of the form $X^l=a$ with $a \in K$. This theory can be found in many books, such as in \cite{LA}. 

Assume now that $l=p$. The situation changes slightly: Artin-Schreier theory tells us that cyclic extensions of degree $p$ are given by equations of the form $X^p-X=a$, where $a \in K$. This theory can also be found in many books, such as in \cite{LA}. It is less well-known that this theory can be extended to the case when $l=p^n$.

Assume $l=p^n$ with $n \geq 1$. The theory of Artin-Schreier-Witt extensions handles this case (\cite{LA}, \cite{WITT}). This theory tells us that all $\Z/p^n\Z$-extensions of $K$ are given by equations of the form $X^p-X=a$,  where $a \in W_n(K)$ is a $p$-typical Witt vector of length $n$ over $K$ (see Section \ref{s2}). Similarly, $\Z_p$-extensions are given by equations of the form $X^p-X=a$,  where $a \in W(K)$ is a $p$-typical Witt vector over $K$.

In Section \ref{s2}, we try to give a complete description of the Artin-Schreier-Witt theory. Most results in this section are known, but the authors could not find a source which contains a similar amount of details. 

\subsection{$\Z_p$-extensions of $k((T))$}

Set $K=k((T))$, where $k$ is a finite field. Local class field theory studies the abelian Galois extensions of $K$. Combining local class field theory and the theory of Artin-Schreier-Witt extensions gives us the so-called $n$th Schmid-Witt pairing (Proposition \ref{711}):
\begin{align*}
[\ ,\ )_n: W_n(K) \times K^* \to W_n(\F_p) =\Z/p^n\Z.
\end{align*}
In Section \ref{s3} we study this symbol.
This pairing has a simple well known formula when $n=1$. In that case $W_1(K)=K$, and one has $$[x,y)_1=\mathrm{Tr}_{k/\F_p} \left( \mathrm{Res}(x dy/y) \right),$$
where $\mathrm{Res}: K \to k$ is the residue map and $d$ is the differentiation map. For $n>1$, and for the limit symbol 
$$[\ ,\ ): W(K) \times K^* \to W(\F_p)=\Z_p,$$
less satisfying formulas were known (for example, Theorem \ref{f1}, \cite{WITT}, \cite{THO}). Our first achievement is a new formula for $[\ ,\ )$, which resembles the simple formula when $n=1$. Let us explain the result. The left kernel of $[\ ,\ )$ is a subgroup called $\wp W(K)$. We first get a better understanding of the abelian quotient group $W(K)/\wp W(K)$. For $a \in K$ we set $[a]=(a,0,0,\ldots) \in W(K)$, the Teichm\"uller lift of $a$. Fix $\alpha \in k$ with $\mathrm{Tr}_{k/\F_p}(\alpha) \neq 0$. Let $\beta=[\alpha] \in W(k)$. We show (Proposition \ref{712}) that any element of $x \in W(K)/\wp W(K)$ has a unique representative of the form
\begin{align*}
x \equiv  c \beta + \sum_{(i,p)=1} c_i [ T ]^{-i} \pmod{\wp W(K)}
\end{align*}
with $c \in W(\F_p)$ and $c_i \in W(k)$ with $c_i \to 0$ as $i \to \infty$. We set 
\begin{align*}
\tilde{x}= c \beta +  \sum_{(i,p)=1} c_i T^{-i} \in W(k)[[T^{-1}]].
\end{align*} 
Any $y \in K^*$ can be written uniquely as 
\begin{align*}
y=a \cdot T^e \cdot \prod_{(i,p)=1}^{\infty}  \prod_{j=0}^{\infty} (1- a_{ij} T^i)^{p^j}
\end{align*}
with $a \in k^*$, $a_{ij} \in k$, $e \in \Z$. 
Set 
\begin{align*}
\tilde{y}= [a] \cdot T^e \cdot \prod_{(i,p)=1}^{\infty}  \prod_{j=0}^{\infty} (1- [a_{ij}] T^i)^{p^j} \in W(k)[[T]].
\end{align*}
We let $\mathrm{Res}$ be the residue of an element in (an enlargement of) $W(k)[[T]]$, and we let $d$ be the derivative operator on $W(k)[[T]]$. One then has the following theorem (Theorem \ref{100}).

\begin{theorem}
\begin{align*}
[x,y) = \mathrm{Tr}_{W(k)/W(\F_p)} \left( \mathrm{Res}( \tilde{x} \cdot d\tilde{y}/\tilde{y} \right).
\end{align*}
\end{theorem}

The simple nature of the above formula, simplifies ramification group and conductor computations (Theorem \ref{up}, Proposition \ref{715}, compare to \cite{THO} and \cite{SCH1}). We also establish a small set of properties which determine the symbol $[\ ,\ )$ completely (Proposition \ref{711}, Proposition \ref{713}).

\subsection{$\Z_p$-extensions of function fields}

Let $K$ be a function field over a finite field $k$ of characteristic $p$. Using the theory from the previous paragraph, one can study $\Z_p$-extensions of $K$. This is the topic of Section \ref{s5}. In particular, using the conductors, one can compute the genus of curves in such towers (Theorem \ref{831}, compare to \cite{SCH1}). In particular, let $K_{\infty}/K$ be a $\Z_p$-extension with $\Z/p^n\Z$ sub-extensions $K_n$. Let $n_u$ be  maximal such that $K_{n_u}/K$ is unramified. Fix $\epsilon >0$. For $n$ large enough, one has (Corollary \ref{15b}), 
$$g(K_n) \geq \frac{p^{2(n-n_u)-1}}{3+\epsilon},$$ 
where $g(K_n)$ denote the genus of $K_n$. 
Our results also show that the previously obtained lower bound for the genus in the literature is incorrect: the $\epsilon$ cannot be dropped (Remark \ref{goa}). We also give criteria for towers for which the genus for large enough $n$ satisfies 
$$g(K_n)=a p^{2n}+b p^n+c$$ 
for some $a,b,c \in \Q$ (Proposition \ref{blaat}). These are the so-called genus stable towers. 
Classification of genus stable $\Z_p$-towers is completed in this paper, which was our 
main motivation. 
We hope that zeros of the zeta function in such a genus stable tower behave in a nice $p$-adic way \cite{WAN7}.

Finally, we specifically discuss $\Z_p$-extensions of $K=k(X)$. Just as in the local case, we manage to find a unique representative of each $\Z_p$-extension of $K$ and compute the conductors in all such towers.

\section{Prerequisites}

\subsection{Witt vectors}

For a detailed description, see \cite{THO}, \cite{RABI}, or follow the exercises from \cite[Chapter VI, Exercises 46-48]{LA}. We will give a brief summary which we will use as a black box. 

Let $p$ be a prime number. Let $R$ be a commutative ring with identity. We define the ring of $p$-typical Witt vectors $W(R)$ as follows. 

\begin{definition}
Let $\mathcal{C}$ be the category of commutative rings with identity. Then there is a unique functor $W: \mathcal{C} \to \mathcal{C}$ such that the following hold:
\begin{itemize}
\item For a commutative ring $R$, one has $W(R)=R^{\Z_{\geq 0}}$ as sets. \\
\item If $f: R \to S$ is a ring morphism, then the induced ring morphism satisfies $W(f) ((r_i)_i)=(f(r_i))_i$. \\
\item The map $g=(g^{(i)})_i: W(R) \to R^{\Z_{\geq 0}}$ defined by
\begin{align*}
(r_i)_i \to \left( \sum_{j=0}^i p^j r_j^{p^{i-j}} \right)_i.
\end{align*}
is a ring morphism (where $R^{\Z_{\geq 0}}$ has the product ring structure).
\end{itemize}
\end{definition}
It turns out that $W(\F_p)=\Z_p$. If $k$ is a finite field, then $W(k)$ is isomorphic to the ring of integers of the unramified field extension of $\Z_p$ with residue field $k$. 

The zero element of $W(R)$ is $(0,0,\ldots)$ and the identity element is $(1,0,0,\ldots)$. The above map $g$ is called the ghost map, and this map is an injection if $p$ is not a zero divisor in $R$. This ghost map, together with functoriality, determines the ring structure.
Furthermore, we have the Teichm\"uller map
\begin{align*}
[\ ]: R \to& W(R) \\
r \mapsto& (r,0,0,\ldots),
\end{align*}
This map is multiplicative: for $r,s \in R$ one has $[rs]=[r][s]$. 
Note that $g([r])=(r, r^p, r^{p^2},\ldots)$. 
We have the so-called Verschiebung group morphism
\begin{align*}
V: W(R) \to& W(R) \\
(r_0,r_1,r_2,\ldots) \mapsto& (0,r_0,r_1,r_2,\ldots). 
\end{align*}
We make $W(R)$ into a topological ring as follows. The open sets around $0$ are the sets of the form $V^i W(R)$. We call this the $V$-adic topology.  With this topology, $W(R)$ is complete and Hausdorff. Furthermore, a ring morphism $R \to S$ induces a continuous map $W(R) \to W(S)$. Any $r=(r_0,r_1,\ldots) \in W(R)$ can be written as $r=\sum_{i=0}^{\infty} V^i[r_i]$. For $n \in \Z_{\geq 0}$, the subgroup $V^n W(R)$ is an ideal, and we can consider the quotient ring of truncated Witt-vectors of length $n$:
\begin{align*}
W_n(R)= W(R)/V^n W(R),
\end{align*}
endowed with the discrete topology. For example, $W_n(\F_p)=\Z/p^n\Z$. 
We write $\pi_n: W(R) \to W_n(R)$ for the natural map (also for $m \geq n$ we consider $\pi_n: W_m(R) \to W_n(R)$). One has $W_1(R)=R$, meaning that in $W(R)$ one has
\begin{align*}
(r_0,\ldots) + (r_0',\ldots) = (r_0+r_0', \ldots), \\
(r_0,\ldots) \cdot (r_0',\ldots) = (r_0 \cdot r_0', \ldots).
\end{align*}
Sometimes we write $W_{\infty}(R)$ instead of $W(R)$. 

Now let us restrict to the case where $R=K$ is a field of characteristic $p$. The ring $W(K)$ has the subring $W(\F_p) = \Z_p$. Witt vectors $(x_0,x_1,\ldots) \in W(K)$ with $x_0 \neq 0$, have a multiplicative inverse (note that $W(K)$ is not a field, since $p$ is not invertible). The Frobenius map $x \mapsto x^p$ on $K$ induces a ring morphism 
\begin{align*}
F: W(K) \to& W(K) \\
(r_0,r_1,\ldots) \mapsto& (r_0^p,r_1^p,\ldots).
\end{align*}
 In fact, one has $V F = FV= \cdot p$ and hence we have an induced ring morphism $F: W_n(K) \to W_n(K)$. One also sees that $W(K)$ is a torsion-free $\Z_p$-module. Similarly, $W_n(K)$ is a torsion-free $\Z/p^n\Z$-module. Let $K'/K$ be a Galois extension and let $g \in G=\Gal(K'/K)$. Then we have a map $W(g): W(K') \to W(K')$. If $K'/K$ is finite, we define the following $W(K)$-linear trace map as follows
\begin{align*}
\mathrm{Tr}_{W(K')/W(K)}: W(K') \to& W(K) \\
x \mapsto& \sum_{g \in G} W(g)x. 
\end{align*}

\subsection{Pontraygin duality}

See \cite{MORR} for more details.

We call a topological group $G$ locally compact if $G$ is Hausdorff and if every point has a compact neighborhood. Let $\mathcal{C}$ be the category of locally compact groups with continuous group morphisms. We will now construct a duality on $\mathcal{C}$. 

We let $S^1 \cong \R/\Z$ be the circle topological group. We define the Pontryagin dual of $G$ to be 
\begin{align*}
G^{\vee} = \Hom_{\mathrm{cont}}(G, S^1),
\end{align*}
where $\mathrm{cont}$ means that the morphisms are continuous. 
We endow $G^{\vee}$ with the compact-open topology, that is, a subbasis of the topology is given by the sets of the form 
\begin{align*}
\{\chi \in G^{\vee}: \chi(K) \subseteq U,\ K \subseteq G \textrm{ compact},\ U \subseteq S^1 \textrm{ open}\}.
\end{align*}
This makes $G$ into a topological group. If $G$ is discrete, then the topology on $G^{\vee}$ agrees with the topology coming from the product $(S^1)^G$. Furthermore, if $G \to H$ is a continuous morphism, then we have an induced continuous morphism $H^{\vee} \to G^{\vee}$. This makes $^{\vee}: \mathcal{C} \to \mathcal{C}$ into a contravariant functor.

\begin{theorem}[Pontryagin Duality]\label{p1}
The functor $^{\vee}: \mathcal{C} \to \mathcal{C}$ is an anti-equivalence of categories with quasi-inverse $^{\vee}$.  It induces anti-equivalences between the subcategories
\begin{align*}
\{ G \in \mathcal{C}: G \textrm{ is } X \} \longleftrightarrow \{ G \in \mathcal{C}: G \textrm{ is } Y\}
\end{align*}
where 
\begin{align*}
(X,Y) \in \{&(\textrm{compact},\textrm{discrete}),  (\textrm{profinite}, \textrm{discrete torsion}), (\textrm{finite}, \textrm{finite}), \\
& (\textrm{finite cyclic}, \textrm{finite cyclic})  \}.
\end{align*}
Furthermore, for $G \in \mathcal{C}$ there is an inclusion reversing bijection
\begin{align*}
\{\textrm{closed subgroups of }G\} \longleftrightarrow& \{ \textrm{closed subgroups of }G^{\vee}\} \\
U \mapsto& U^{\perp}=\{\chi \in G^{\vee}: U \subseteq \mathrm{ker}(\chi) \} \\
V^{\perp}=\bigcap_{\chi \in V} \mathrm{ker}(\chi) \mapsfrom& V.
\end{align*}

\end{theorem}

One has the natural isomorphism
\begin{align*}
G \to& \left( G^{\vee} \right)^{\vee} \\
g \mapsto& \left( \psi \mapsto \psi(g) \right).
\end{align*}

\section{$\Z_p$-extensions in characteristic $p$} \label{s2}

Fix a prime $p$ and let $K$ be a field of characteristic $p$. Note that $\Z/p\Z$-extensions of $K$ are given by solving equations of the form $X^p-X-a=0$ with $a \in K$ by Artin-Schreier theory. Below we develop Artin-Scheier-Witt theory which allows us to study $\Z/p^n\Z$-extensions for any $n$.

\subsection{Artin-Schreier-Witt extensions}

Let $K^{\mathrm{sep}}$ be a separable closure of $K$. Let $n \in \Z_{\geq 1}$. We define a group morphism 
\begin{align*}
\wp=F-\mathrm{id}: W_n(K^{\mathrm{sep}}) \to& W_n(K^{\mathrm{sep}}) \\
x \mapsto& Fx-x,
\end{align*}
with kernel $W_n(\F_p)$ (we suppress the $n$ in the notion for $\wp$). This map is surjective (one can see this by using induction on the length $n$).  For $a \in W_n(K)$ and $x=(x_0,\ldots,x_{n-1}) \in \wp^{-1}a \subset W_n(K^{\mathrm{sep}})$,  we set $K(\wp^{-1} a)=K(x_0,\ldots,x_{n-1})$. This extension does not depend on the choice of $x$. In fact, $K(\wp^{-1}a)=K(\wp^{-1}b)$ if $a \equiv b \pmod{ \wp W_n(K)}$. For a subset $B \subseteq W_n(K)$ we set $K(\wp^{-1}B)=K(\wp^{-1}b: b \in B)$. We similarly define $K(\wp^{-1}B)$ if $B$ is a subset of $W_n(K)/\wp W_n(K)$. Finally, we consider the projection morphisms $\pi_i: W_n(K)/\wp W_n(K) \to W_i(K)/\wp W_i(K)$ for $i  \leq n$. 
 
\begin{proposition}[Hilbert 90] \label{90}
Let $K'/K$ be a cyclic extension of degree $m$ with Galois group $G=\langle \sigma \rangle$. Then we have a short exact sequence
\begin{align*}
0 \to  W_n(K')/W_n(K) \overset{a \mapsto a-\sigma a}{\to}  W_n(K') \overset{\mathrm{Tr}_{W_n(K')/W_n(K)}}{\to} W_n(K) \to 0
\end{align*}
\end{proposition}
\begin{proof}
Similar to the proof of \cite[Chapter VI, Theorem 6.3]{LA}.
\end{proof}

\begin{theorem} \label{11}
Let $K$ be a field of characteristic $p$. 
\begin{enumerate}
\item
Let $K'/K$ be a non-trivial cyclic Galois extension of degree $p^n$ with group $G=\langle \sigma \rangle$. Then there is an $a \in W_n(K)$ with $K'=K(\wp^{-1}a)$, and $\pi_1(a) \not \in \wp K$. 
\item
Let $a \in W_n(K)/\wp W_n(K)$. Then $K'=K(\wp^{-1}a)$ is Galois over $K$ and the maps 
\begin{align*}
\psi_1: \Gal(K'/K) \to& \Hom(\langle a \rangle, W_n(\F_p)) \\
g \mapsto&  \left(\wp b \mapsto g b- b \right)
\end{align*}
and
\begin{align*}
\psi_2: \langle a \rangle \to&  \Hom(\Gal(K'/K),W_n(\F_p)) \\
\wp b \mapsto& \left(g \mapsto gb-b \right)
\end{align*}
are well-defined isomorphisms. Set $r=\max\{m \in \{0,1,\ldots,n\}: \pi_m(a)=0\}$ and $s=n-r$.
Then one has (non-canonical) isomorphisms
\begin{align*}
\Z/p^s\Z \cong \Gal(K'/K)  \cong \Hom(\langle a \rangle, W_n(\F_p)) \cong \langle a \rangle \cong \Hom(\Gal(K'/K),W_n(\F_p)) .
\end{align*}
\end{enumerate}
\end{theorem}
\begin{proof}
i: Consider the Witt vector $1 \in W_n(K)$ with $\mathrm{Tr}_{W_n(K')/W_n(K)}(1)=p^n=0$. By Theorem \ref{90} there is an $\alpha=(a_0,\ldots,a_{n-1}) \in W_n(K')$ with $1= \alpha - \sigma \alpha$. One finds $K'=K(\alpha)=K(a_0,\ldots,a_{n-1})$. Furthermore, one has 
\begin{align*}
\sigma \wp \alpha = \sigma (F \alpha - \alpha) = F( \sigma \alpha) - \sigma \alpha=F(\alpha+1)-\alpha+1 = F \alpha - \alpha = \wp \alpha.
\end{align*}
Hence we have $\wp \alpha \in W_n(K)$. The last statement follows from the proof in ii.

ii: Note that on $W_n(K)/\wp W_n(K)$, $\cdot p=V$. Hence one can read off the order of $a$ as stated, and we find $\langle a \rangle \cong \Hom(\langle a \rangle, W_n(\F_p)) \cong \Z/p^s\Z$.  It is obvious that $K(\wp^{-1} a)/K$ is Galois, say with group $G$, and that $\psi_1$ is a well-defined injective group morphism. If $r=n$, the result follows easily. Assume $r<n$. Consider the following commutative diagram:
\begin{displaymath}
    \xymatrix{ G \ar@{^(->}[r] \ar@{->>}[d] & \Hom(\langle a \rangle, W_n(\F_p)) \cong \Z/p^s\Z \ar@{->>}[d] \\
               G_{r+1}=\Gal(K(\wp^{-1} \pi_{r+1}(a))/K) \ar@{^(->}[r]&  \Hom(\langle \pi_{r+1}(a) \rangle, W_{r+1}(\F_p)) \cong \Z/p\Z.  }
\end{displaymath}
We claim that the second horizontal map is an isomorphism. Indeed, the extension $K(\wp^{-1} \pi_{r+1}(a))/K$ is given by a single Artin-Schreier equation of the form $x^p-x-b$ which has no roots in $K$ and hence has degree $p$ over $K$. Since $G_{r+1}$ and $\Hom(\langle \pi_{r+1}(a) \rangle, W_{r+1}(\F_p))$ have the same order, the map is an isomorphism. Hence the map $G \to \Z/p\Z$ is surjective. The map $\Z/p^s\Z \to \Z/p\Z$ is the natural projection map, and it follows that the map $G \to \Z/p^s\Z$ is surjective. We conclude that the first horizontal map is an isomorphism. 

The map $\psi_2$ is the Pontryagin dual of the isomorphism $\psi_1$, and hence it is an isomorphism (Theorem \ref{p1}).
\end{proof}

In fact, one easily sees if $\wp^{-1}a=(b_0,b_1,\ldots,b_{n-1})$, then $K(\wp^{-1}a)=K(b_{n-1})$. Indeed, $b_{n-1}$ does not belong to any of the subfields of the extension.

\begin{remark}
Similarly, one can study $W(\F_p)^*=\Z_p^*$-extensions by considering $W(K)^*$ and using the map $x \mapsto Fx/x$. 
\end{remark}

\subsection{Artin-Schreier-Witt theory}

We will now give a more topological interpretation of the above results, with the help of Pontryagin duality. This also allows us to understand any abelian extension of exponent $p^n$. 

Let $M$ be a $\Z_p$-module. We make $M$ into a topological group by picking a basis around $0$ consisting of the sets of the form $p^i M$. We call this topology the $p$-adic topology. Any group morphism between $\Z_p$-modules is automatically continuous. 
Let $H$ be an abelian group. We define its $p$-adic completion, a $\Z_p$-module, to be 
\begin{align*}
\widehat{H}=\underset{\underset{i}{\leftarrow}}\lim\ H/p^iH.
\end{align*}
Let $M/K$ be an Galois extension with Galois group $H=\Gal(M/K)$. We endow $H$ with the Krull topology as follows. If $L/K$ is finite Galois subextension, we endow $\Gal(L/K)$ with the discrete topology. We endow 
\begin{align*}
H= \underset{\underset {L:\ L/K \textrm{ finite}}  \leftarrow}\lim  \Gal(L/K)
\end{align*}
with the induced topology coming from the product topology on the product group $\prod_{L:\  L/K \textrm{ finite}} \Gal(L/K)$. This makes $H$ into a compact topological group. If $H$ is abelian, we consider the quotient 
\begin{align*}
H_p= \underset{\underset {L:\ L/K \textrm{ finite } p-\textrm{power order}}  \leftarrow}\lim  \Gal(L/K)= \underset{\underset {i}  \leftarrow}\lim\ H/p^iH,
\end{align*}
with the Krull topology (this does not coincide with the $p$-adic topology usually, see Remark \ref{r1} below).

Let $n \in \Z_{\geq 1} \sqcup \{\infty\}$. We endow $W_n(K)$ with the $V$-adic topology, that is, a basis of open sets around $0$ are of the form $V^i W_n(K)$. If $K$ is perfect, one has $V^i W_n(K)=p^i W_n(K)$ and this is just the $p$-adic topology. If $n$ is finite, $W_n(K)$ has the discrete topology. Note that $\wp W_n(K)$ is a closed subgroup. The group $W_n(K)/\wp W_n(K)$ becomes a topological group. Note that $V=VF=\cdot p$ on the quotient, and hence the $\Z_p$-module $W_n(K)/\wp W_n(K)$ always has the $p$-adic topology.

\begin{theorem} \label{15}
Let $K$ be a field of characteristic $p$ with group $G=\Gal(K^{\mathrm{ab}}/K)$. Let $n \in \Z_{\geq 1}$. Then we have a well-defined perfect pairing
\begin{align*}
(\ ,\ ): G/p^nG \times W_n(K)/\wp W_n(K) \to& W_n(\F_p) \\
(g, \wp a \pmod{W_n(K)}) \mapsto& g a- a 
\end{align*}
and the maps
\begin{align*}
\varphi_1: G/p^nG \to& \Hom(W_n(K)/\wp W_n(K) , W_n(\F_p)) \\
g \mapsto& \left( \wp a \pmod{W_n(K)} \mapsto g a- a \right)
\end{align*}
and
\begin{align*}
\varphi_2: W_n(K)/\wp W_n(K) \to& \Hom_{\mathrm{cont}}(G/p^nG,W_n(\F_p)) \\
\wp a \mapsto& \left( g \mapsto ga- a \right)
\end{align*}
are well-defined homeomorphism of topological groups. The same statements remain true if we set $n = \infty$ and replace $G/p^nG$ by $G_p$. 
\end{theorem}
\begin{proof} 
It is not hard to see that $(\ ,\ )$ is well-defined group-morphism and by Theorem \ref{11}, the pairing is non-degenerate, that is, the left and right kernel are trivial. 

Now let $n \in \Z_{\geq 1}$. It is not hard to see that the pairing induces an injective group morphism $\varphi_1$. We will now show that the map is surjective. Let $\chi \in \Hom_{\mathrm{cont}}(G/p^nG,W_n(\F_p))$. By Theorem \ref{11} it follows that there is an $a \in W_n(K)/\wp W_n(K)$ such that $\chi$ factors through 
\begin{align*}
\chi: G/p^nG \to G'=\Gal(K(\wp^{-1}a)/K) \to W_n(\F_p)
\end{align*}
(Theorem \ref{11}i). By Theorem \ref{11}ii we have an isomorphism $\langle a \rangle \cong \Hom(G',W_n(\F_p))$. Hence there is a multiple of $a$ which maps to $\chi$.
Note that $W_n(K)/\wp W_n(K) $ has the discrete topology, and so has $(G/p^nG)^{\vee}$ by Pontryagin Duality (Theorem \ref{p1}). Hence $\varphi_1$ is a homeomorphism. By Theorem \ref{p1}) the dual map $\varphi_2$ is a homeomorphism. 

Taking the projective limit, gives isomorphisms
\begin{align*}
W(K)/\wp W(K) \cong& \lim_{\underset{n}{\leftarrow}} W_n(K)/\wp W_n(K) \cong \lim_{\underset{n}{\leftarrow}} \Hom_{\mathrm{cont}}(G/p^nG,W_n(\F_p)) \\
\cong& \Hom_{\mathrm{cont}}(G_p,W(\F_p))
\end{align*}
and
\begin{align*}
G_p \cong \lim_{\underset{n}{\leftarrow}} \Hom( W_n(K)/\wp W_n(K), W_n(\F_p)) \cong \Hom ( W(K)/\wp W(K), W(\F_p)).
\end{align*}
It is not hard to see that these maps are in fact homeomorphisms.
\end{proof}

\begin{remark}
We have chosen not to give a cohomological approach. Let us briefly sketch the cohomological approach. Consider the exact sequence 
$$0 \to W_n(\F_p) \to W_n(K^{\mathrm{sep}}) \overset{\wp}{\to} W_n(K^{\mathrm{sep}}) \to 0$$
of $G=\Gal(K^{\mathrm{sep}}/K)$-modules. Taking $G$-invariants gives us an exact sequence
\begin{align*}
0 \to W_n(\F_p) \to W_n(K) \overset{\wp}{\to} W_n(K) \to H^1(G, W_n(\F_p)) \to H^1(G,W_n(K)) \to \ldots
\end{align*}
Since $W_n(\F_p)$ has trivial $G$-action, $H^1(G, W_n(\F_p))=\Hom_{\mathrm{cont}}(G,W_n(\F_p))$. We claim $H^1(G,W_n(K))=0$.  This follows essentially from Proposition \ref{90}. Hence we obtain $W_n(K)/\wp W_n(K) \cong \Hom_{\mathrm{cont}}(G,W_n(\F_p))$.
\end{remark}

One can then apply the Pontryagin duality theorem to prove the following result. We fix a separable  closure $\overline{K}$ of $K$. 

\begin{theorem} \label{1800}
Let $K$ be a field of characteristic $p$. Let $n \in \Z_{\geq 1}$. Then there is an inclusion preserving bijection
\begin{align*}
\left\{ \textrm{subgroups of }W_n(K)/\wp W_n(K) \right\} \longleftrightarrow& \left\{ K\subseteq L \subseteq \overline{K}: \Gal(L/K) \textrm{ is a }\Z/p^n\Z\textrm{-module}\right\}\\
B \mapsto& K(\wp^{-1}B) \\
\wp W_n(L) \cap W_n(K)/\wp W_n(K) \mapsfrom& L.
\end{align*}
Furthermore, for a subgroup $B$ of $W_n(K)/\wp W_n(K)$ we have a natural commutative diagram 
\begin{align*}
    \xymatrix{ G/p^nG \ar[r]^{\sim \ \ \ } \ar@{->>}[d] & \ar@{->>}[d] \Hom(W_n(K)/\wp W_n(K),W_n(\F_p)) \\
               \Gal(K(\wp^{-1}B)/K) \ar[r]^{\sim} & \Hom(B,W_n(\F_p)). }
\end{align*}
If $B$ is finite, one has $|B|=[K(\wp^{-1}B):K]$. Finally, for $0 \leq i \leq n$ one has, where $H=\Gal(K(\wp^{-1}B)/K)$:
\begin{align*}
    \xymatrix{ K(\wp^{-1}B) \ar@{-}^{p^iH}[d]  \\
               K \left(\wp^{-1} \left( B \cap p^{n-i} (W_n(K)/\wp W_n(K)) \right) \right)  \ar@{-}^{H/p^iH}[d] \\
K.
 }
\end{align*}

\end{theorem}
\begin{proof}
Let $K_n$ be the maximal abelian extension of exponent $p^n$. By Theorem \ref{p1}, (closed) subgroups of $W_n(K)/\wp W_n(K)$ correspond to closed subgroups of $G/p^nG$. In this case, $B$ corresponds to $\{g \in G/p^nG: gb=b, b \in \wp^{-1}B\}=\Gal(K_n/K(\wp^{-1}B))$ and hence to $K(\wp^{-1}B)$ by Galois theory. This gives us the inclusion preserving bijection. It is not hard to see that an extension $L/K$ of exponent dividing $p^n$, corresponds to $\wp W_n(L) \cap W_n(K)/\wp W_n(K)$. 

Let $L=K(\wp^{-1}B)$. We have a natural injective map 
\begin{align*}
B \to \Hom_{\mathrm{cont}}(\Gal(L/K),W_n(\F_p))
\end{align*}
between two groups with the discrete topology. Take $\chi \in \Hom_{\mathrm{cont}}(\Gal(L/K),W_n(\F_p))$. Then $\chi$ factors as 
\begin{align*}
\chi: \Gal(L/K) \to \Gal(K(\wp^{-1}a)/K) \to W_n(\F_p)
\end{align*}
for some $a \in W_n(K)$ (Theorem \ref{11}). By the previous statement, it follows $a \in B$. By Theorem \ref{15}ii we have an isomorphism $\langle a \rangle \cong \Hom(G',W_n(\F_p))$ and hence there is a multiple of $a$ which maps to $\chi$.  This shows that the map is surjective and hence a homeomorphism. By Pontryagin duality we see that the map $\Gal(L/K) \to \Hom(B,W_n(\F_p))$ is a homeomorphism as required. It is easy to see that the diagram commutes. The first vertical map is surjective, and we have seen before that the horizontal maps are isomorphisms (Theorem \ref{15}). Hence it follows that the right vertical map is surjective. The statement about the cardinalities now follows easily. 

The last statement follows from Theorem \ref{11}. 
\end{proof}

Since $V$ commutes with $\wp$, the injective group morphism $V: W_i(K) \to W_{i+1}(K)$ induces an injective group morphism $V: W_i(K)/\wp W_i(K) \to W_{i+1}(K)/\wp W_{i+1}(K)$ with image $p \left(W_{i+1}(K)/\wp W_{i+1}(K) \right)$. 

\begin{corollary} \label{1230}
Let $K$ be a field of characteristic $p$. 
Let $S$ be the set of sequences $(B_1,B_2,\ldots)$ with $B_i \subseteq W_i(K)/\wp W_i(K)$ a subgroup with $VB_i=B_{i+1} \cap p \left( W_{i+1}(K)/\wp W_{i+1}(K) \right)$. We say that $(B_1,B_2,\ldots) \subseteq (B_1',B_2',\ldots)$ if $B_i \subseteq B_i'$ for all $i$. Then there is an inclusion preserving bijection 
\begin{align*}
S \longleftrightarrow&  \left\{K \subseteq L \subseteq \overline{K}: \Gal(L/K) \textrm{ is a } \Z_p\textrm{-module}\right\} \\
(B_1,B_2,\ldots) \mapsto& \underset{\underset{i}{\rightarrow}}\lim\ K(\wp^{-1}B_i)= K(\wp^{-1}B_1,\wp^{-1}B_2,\ldots) \\
\left(\wp W_i(L) \cap W_i(K)/\wp W(K)\right)_i \mapsfrom& L.
\end{align*}
Furthermore, for $(B_i)_i \in S$ we have a commutative diagram for $n \in \Z_{\geq 1}$: 
\begin{align*}
    \xymatrix{ G_p \ar[r]^{\sim\ \ \ } \ar@{->>}[d] & \ar@{->>}[d] \Hom(W(K)/\wp W(K),W(\F_p)) \\
               H=\Gal(\underset{\underset{i}{\rightarrow}}\lim\ K(\wp^{-1}B_i)/K) \ar[r]^{\sim} \ar@{->>}[d] & \underset{\underset{i}{\leftarrow}}\lim \Hom(B_i,W_i(\F_p)) \ar@{->>}[d] \\
H/p^iH=\Gal(K(\wp^{-1}B_i)/K) \ar[r]^{\sim} & \Hom(B_i ,W_i(\F_p)). }
\end{align*}
Here $\underset{\underset{i}{\leftarrow}}\Hom(B_i,W_i(\F_p))$ is the set of morphisms $(\psi_i)_i$ such that all diagrams
\begin{align*}
    \xymatrix{ B_i \ar[r]^{\psi_i} \ar[d]^{V} & W_i(\F_p) \ar[d]^{V} \\
              B_{i+1} \ar[r]^{\psi_{i+1}} & W_{i+1}(\F_p). }
\end{align*}
commute.
\end{corollary}
\begin{proof}
The map is well-defined by Theorem \ref{15} and one has $\underset{\underset{i}{\rightarrow}}\lim\ K(\wp^{-1}B_i)= K(\wp^{-1}B_1,\wp^{-1}B_2,\ldots)$. By Theorem \ref{1800} the map is a bijection and all diagrams follow. 
\end{proof}

\begin{example}
Let us study a couple of specific cases of Corollary \ref{1230}. Let $L/K$ be a Galois extension such that $G=\Gal(L/K)$ is a $\Z_p$-module with corresponding groups $(B_1,B_2,\ldots)$. Below we study the cases when $G$ is $\Z/p^n\Z$-torsion, torsion-free, finitely generated and torsion.

First, $G$ is a $\Z/p^n\Z$-module if and only if $B_{n+i}=V^i B_{n}$ for all $i \geq 0$ (equivalently, just for $i=1$). In that case $L=K(\wp^{-1}B_n)$ and 
\begin{align*}
\underset{\underset{i}{\leftarrow}}\lim \Hom(B_i,W_i(\F_p))=\Hom(B_n,W_n(\F_p))
\end{align*}
and we recover Theorem \ref{1800}. 

Secondly, $G$ is torsion-free if and only if all maps $\pi_i: B_{i+1} \to B_{i}$ are surjective (meaning that each $\Z/p^i\Z$-extension extends to a $\Z/p^{i+1}\Z$-extension). In that case we can form a group $B=\underset{\underset{i}{\leftarrow}}\lim\ B_i \subseteq W(K)/\wp W(K)$ with $\pi_i(B)=B_i$. One has 
\begin{align*}
B \cap p \left(W(K)/\wp W(K)\right) = p B.
\end{align*}
It is not hard to see that we have an isomorphism of topological groups
\begin{align*}
\underset{\underset{i}{\leftarrow}}\lim \Hom(B_i,W_i(\F_p))=&\Hom(B,W(\F_p)) \\
(\psi_i)_i \mapsto& \left( b \mapsto \lim_{i \to \infty} \psi_i(\pi_i b) \right).
\end{align*}
Conversely, a group $B \subseteq  W(K)/\wp W(K)$ with 
\begin{align*}
B \cap p \left(W(K)/\wp W(K)\right) = p B
\end{align*}
gives rise to a sequence $(\pi_1(B),\pi_2(B),\ldots) \in S$ with $B=\underset{\underset{i}{\leftarrow}}\lim\ \pi_i(B)$.

Thirdly, $G$ is finitely generated if and only if $\{ \dim_{\F_p} B_i/p B_i: i \}$ has a maximum, and this maximum is the minimal number of generators. 

Finally, $G$ is torsion if and only if for all $i$ and $x \in B_i$ there is a $j \geq i$ with $x \not \in \pi_i(B_j)$ with $\pi_i: W_j(K)/\wp W_j(K) \to W_i(K)/\wp W_i(K)$ the natural map. 

Consider a sequence $(B_1,B_2,\ldots)$ with $B_i \subseteq W_i(K)/\wp W_i(K)$, not necessarily in $S$. The unique sequence of subgroups corresponding to $K(\wp^{-1}B_1,\wp^{-1}B_2,\ldots)$ is given by $(C_1, C_2,\ldots)$ with
\begin{align*}
C_n =& B_n + V^{-1}(B_{n+1} \cap p \left( W_{n+1}(K)/\wp W_{n+1}(K)  \right) \\
&+ V^{-2}(B_{n+2} \cap p^2 \left( W_{n+2}(K)/\wp W_{n+2}(K)  \right) +  \ldots.
\end{align*} 

\end{example}

Let $R$ be a commutative ring. We say that an $R$-module $M$ is torsion-free if for every $m \in M$, $m \neq 0$, the set $\{r \in R: rm=0\}$ is equal to $0$. 

\begin{lemma} \label{184}
Let $n \in \Z_{\geq 1} \sqcup \{\infty\}$. Let $M$ be a torsion-free $W(\F_p)$-module with $M=\widehat{M}$. Let $\mathfrak{B}$ be a basis of $M/pM$ over $\F_p$. Then the map
\begin{align*}
\widehat{ \bigoplus_{\mathfrak{B}} W_n(\F_p)} \to M/p^nM \\
\left( c_{\overline{b}} \right)_{\overline{b} \in \mathcal{B}} \mapsto \sum_{\overline{b} \in \mathcal{B}} c_{\overline{b}}b
\end{align*}
is a homeomorphism of topological groups, where $M/p^{\infty}M=M$ by definition.
\end{lemma}
\begin{proof}
Since both groups have the $p$-adic topology, it is enough to show that the map is a bijection. We give a proof for $n=\infty$, the finite 
$n$ case is a corollary.

Assume that $\sum_{\overline{b} \in \mathcal{B}} c_{\overline{b}}b=0$ is a non-trivial combination. Let $m$ be the minimal valuation of the $c_{\overline{b}}$. Modulo $pM$ one obtains
\begin{align*}
\sum_{\overline{b} \in \mathcal{B}} \overline{c_{\overline{b}}/p^m} \overline{b}=0,
\end{align*}
which is a contradiction with the independence of the $\overline{b}$. This shows that the map is injective.

For the surjectivity, let $x \in M$. Then one can write $\overline{x} = \sum_{\overline{b} \in \mathcal{B}} d_{\overline{b}} \overline{b}$. Hence $x -  \sum_{\overline{b} \in \mathcal{B}} [d_{\overline{b}} ] b \in p M$. Since $M$ is complete, we can continue this procedure. This shows that the map is surjective.
\end{proof}

\begin{proposition} \label{182}
Let $K$ be a field of characteristic $p$ and let $n \in \Z_{\geq 1} \sqcup \{\infty\}$. Let $\mathfrak{B}$ be a basis of $K/\wp K$ over $\F_p$. Then the map
\begin{align*}
\widehat{ \bigoplus_{\mathfrak{B}} W_n(\F_p)} \to W_n(K)/\wp W_n(K) \\
(a_b)_{b \in \mathfrak{B}} \to \sum_{i} a_b [b] \pmod{\wp W(K)}
\end{align*}
is an isomorphism of topological groups. 
\end{proposition}
\begin{proof}
We prove the case $n=\infty$. The other cases are similar. 
Set $H=W(K)/\wp W(K)$. On $H$, one has $\cdot p=V^1$. This shows that $H$ is a complete torsion-free $W(\F_p)$-module. One has $\wp W(K) + p^i W(K) = V^i W(K) + \wp W(K)$.  
This gives
\begin{align*} 
H/pH \cong  W(K)/(\wp W(K) + p W(K)) = W(K)/(V^1 W(K) + \wp W(K))  \cong K/\wp K.
\end{align*}
The result follows from Lemma \ref{184}. 

\end{proof}

\begin{corollary}
Let $K$ be a field of characteristic $p$ and let $n \in \Z_{\geq 1} \sqcup \{\infty\}$. Let $\mathfrak{B}$ be a basis of $K/\wp K$ over $\F_p$ and for $b \in \mathfrak{B}$ let $b' \in \wp^{-1} [b] \subset W_n(K^{\mathrm{sep}})$. Then the map
\begin{align*}
G/p^nG \to& \prod_{\mathfrak{B}} W_n(\F_p) \\
g \mapsto& ( g b' - b')_{b \in \mathfrak{B}}
\end{align*}
is an isomorphism of topological groups. The statement is also true when $n=\infty$ and $G/p^nG$ is replaced by $G_p$. 
\end{corollary}
\begin{proof}
By Proposition \ref{182} and Theorem \ref{15}, we have for finite $n$
\begin{align*}
G/p^nG \cong \Hom( \widehat{ \bigoplus_{\mathfrak{B}}W_n(\F_p)}, W_n(\F_p) ). 
\end{align*}
We have a natural homeomorphism
\begin{align*}
\Hom(\widehat{\bigoplus_{\mathfrak{B}} W_n(\F_p)}, W_n(\F_p)) \to \prod_{\mathfrak{B}} W_n(\F_p).
\end{align*}
The same proof works for $n=\infty$ and $G_p$. 
\end{proof}

\begin{remark} \label{r1}
By the above, one has $G_p \cong \prod_{\mathfrak{B}} \Z_p$ where $\mathfrak{B}$ is a basis of $K/\wp K$. It follows that $G_p$ has the $p$-adic topology if and only if $\mathfrak{B}$ is finite. 
\end{remark}

\begin{example} \label{183}
In certain cases, one can easily find a basis of $K/\wp K$ over $\F_p$. Below we will construct a subset $\mathcal{D}$ of $K$ which injects into $K/\wp K$ and such that its image forms an $\F_p$-basis of $K/\wp K$. 
For this it is enough to show that $\mathrm{Span}_{\F_p}(\mathcal{D}) \cap \wp K = 0$ and $\mathrm{Span}_{\F_p}(\mathcal{D}) + \wp K=K$. 
\begin{itemize}
\item Assume $K=K^{\mathrm{sep}}$. Then one can take $\mathcal{D}=\emptyset$. 
\item Assume $K$ is a finite field. Take any vector $b$ with $b \not \in \wp K$, that is, take any $b \in K$ with $\mathrm{Tr}_{K/\F_p}(b) \neq 0$. One can take $\mathcal{D}=\{b\}$ (Proposition \ref{90}).
\item Assume $K=k((T))$ for a perfect  field $k$. Let $\mathfrak{B}$ be a subset of $k$ giving a basis of $k/\wp k$ over $\F_p$ and let $\mathfrak{C}$ be a basis of $k$ over $\F_p$. Then one can take
\begin{align*}
\mathcal{D}= \mathfrak{B} \sqcup \{ c T^{-i}: c \in \mathfrak{C},\ (i,p)=1, i\geq 1\}.
\end{align*}
Let us prove this result. If $f \in Tk[[T]]$, then set $g=- \sum_{i=0}^{\infty} f^{p^i} \in Tk[[T]]$. One has $\wp g =f$. Note that $a_iT^{-ip} \equiv a_i^{1/p} T^{-i} \pmod{ \wp K}$ (where we use that $k$ is perfect). Hence we find $\mathrm{Span}_{\F_p}(\mathcal{D}) + \wp K=K$.  Let $f = \sum_i a_i T^i$. One has $\wp f=  \sum_i a_i^p T^{ip}- \sum_i a_i T^i$. We find $\mathrm{Span}_{\F_p}(\mathcal{D}) \cap \wp K = 0$.
\item Assume $K=k(X)$ for some perfect field $k$. Let $\mathfrak{B}$ be a subset of $k$ giving a basis of $k/\wp k$ over $\F_p$ and let $\mathfrak{C}$ be a basis of $k$ over $\F_p$. Then one can take
\begin{align*}
\mathcal{D}= \mathfrak{B} \sqcup& \bigsqcup_f \{ \frac{bX^i}{f^j}:\ (j,p)=1,\ j\geq 1, \ 0 \leq i< \deg(f),\ b \in \mathcal{C} \} \\ \sqcup& \{ b X^j,\ (j,p)=1,\ j\geq 1, \ b \in \mathcal{C} \},
\end{align*}
where $f\in k[x]$ runs over monic irreducible polynomials. 
Indeed, by partial fractions, one can write any $a \in K$ uniquely as a finite sum
\begin{align*}
a= a_0+ \sum_{i=1}^m a_iX^i + \sum_{f} \sum_{i=1}^{\infty} \frac{g_{f,i}}{f^i}
\end{align*}
with $a_i \in k$, $g_{f,i} \in k[X]$ with $0 \leq \deg(g_{f,i})<\deg(f)$ (\cite[Chapter IV, Section 5]{LA}). One has $\mathrm{Span}_{\F_p}(\mathcal{D}) + \wp K=K$. Indeed, a term like $g/f^{p^j}$ can be replaced as follows. One has $g=h_1^p+fh_2$ for some polynomials $h_1, h_2$, because $k[X]/(f)$ is perfect. Then one has
\begin{align*}
g_f/f^{p^j}  = (h_1/f^{p^{j-1}})^{p}+ h_2/f^{p^{j}-1} \equiv h_1/f^{p^{j-1}}+ h_2/f^{p^{j-1}} \pmod{\wp K}. 
\end{align*}
From the expression of $a$, one easily obtains $\mathrm{Span}_{\F_p}(\mathcal{D}) \cap \wp K = 0$. See Lemma \ref{hulu} for a different way of representing $W(K)/\wp W(K)$. 
\end{itemize}
\end{example}

\section{$\Z_p$-extensions of $k((T))$} \label{s3}

Let $k$ be a finite field of cardinality $q=p^m$ and characteristic $p$. Let $K=k((T))$. The field $K$ has a natural valuation. If $f=\sum_{i \geq v} a_i T^i$ with $a_v \neq 0$, then the valuation is $v$. Elements of $K$ with valuation $1$, such as $T$, are called uniformizers. We set $\pa=Tk[[T]]$, the unique maximal ideal of $k[[T]]$. 

\subsection{Construction of $[\ ,\ )$}
Let $K^{ab}$ be the maximal abelian extension of $K$. Let $G=\Gal(K^{ab}/K)$. Let $G_p= \underset{\underset{n}{\leftarrow}}\lim\ G/p^nG$, endowed with the Krull topology. Set $\widehat{K^*}= \underset{\underset{n}{\leftarrow}}{\lim}\ K^*/(K^*)^{p^n}$, the $p$-adic completion of $K^*$ with its natural $p$-adic topology. Note that $\widehat{K^*} \cong T^{\Z_p} \times U^1$ where $U^1=1+Tk[[T]]$ are the one units of $K$ (the $p$-adic topology on $U^1$ does not coincide with the topology coming from the valuation on $K$). We usually identify $\widehat{K^*}$ with $T^{\Z_p} \times U^1$. We have a natural map $K^* \to \widehat{K^*}$, with kernel $k^*$. 

The Artin map (or Artin reciprocity law) from Class field theory is a certain group morphism $K^* \to G$ (see \cite{SE1}). This map is usually the best way to understand the group $G$ and to understand ramification in abelian extensions of $K$.
This Artin map induces a homeomorphism
\begin{align*}
\psi: \widehat{K^*} \to G_p.
\end{align*}
Theorem \ref{15} gives a homeomorphism $G_p \cong \Hom(W(K)/\wp W(K) , W(\F_p))$. Note that $W(K)/\wp W(K)$ is a $W(\F_p)$-module. If we combine both maps, we obtain a $\Z_p$-bilinear, hence continuous, symbol
\begin{align*}
[\ ,\ ): W(K)/\wp W(K) \times \widehat{K^*} \to& W(\F_p) =\Z_p \\
(\wp x,y) \mapsto& \psi(y)x-x.
\end{align*}
This symbol is often called the Schmid-Witt symbol. The main problem of this section is to find a useful explicit formula for computing $[\ ,\ )$ efficiently.

\begin{remark}
Most text on this topic consider the following symbol for $n \in \Z_{\geq 1}$:
\begin{align*}
[\ ,\ )_n: W_n(K)/\wp W_n(K) \times K^*/(K^*)^{p^n} \to& W_n(\F_p) = \Z/p^n\Z \\
[\wp x,y)_n =&  \pi_n \left( \psi(y)x-x \right).
\end{align*}
We will focus on the symbol $[\ ,\ )$, because the symbols $[\ ,\ )_n$ can be computed from this symbol: $[\pi_n x,y)_n=\pi_n [x,y)$ for $x \in W(K)$ and $y \in K^*$. Conversely, the symbol $[\ ,\ )$ can be obtained from the symbols $[\ ,\ )_n$ by taking a projective limit.
\end{remark}

\subsection{Defining properties of $[\ ,\ )$}
We will now state some properties of $[\ ,\ )$. 

\begin{proposition} \label{711}
The symbol $[\ ,\ )$ has the following properties:
\begin{enumerate}
\item the symbol is $\Z_p$-bilinear;
\item for $a \in W(k)$ and $T' \in K^*$ a uniformizer one has $[a,T')=\mathrm{Tr}_{W(k)/W(\F_p)}(a)$;
\item for a uniformizer $T'  \in K^*$, $c \in k$ and $i>0$ with $p \nmid i$ one has $[[cT'^{-i}],T')=0$.  
\end{enumerate}
\end{proposition}
\begin{proof}
i: Follows from the construction. \\

ii: The extension corresponding to $a$ is unramified. Hence by class field theory, a uniformizer $T'$ is mapped to the Frobenius element. Assume $\wp b=a$. Then one has
\begin{align*}
[a,T') = F^mb-b= \sum_{i=0}^{m-1} F^i(Fb-b)=\sum_{i=0}^{m-1} F^ia=\mathrm{Tr}_{W(k)/W(\F_p)}(a).
\end{align*}

iii: This is the main theorem in \cite{TEIC}.
\end{proof}

Let $\alpha \in k$ with $\mathrm{Tr}_{k/\F_p}(\alpha) \neq 0$. Set $\beta=[\alpha] \in W(k) \subset W(K)$. We will now discuss the structure of $W(K)/\wp W(K)$. 

\begin{proposition} \label{712}
Let $T'$ be a uniformizer of $K$. Then any $x=(x_0,x_1,\ldots) \in W(K)$ has a unique representative in $W(K)/\wp W(K)$ of the form 
\begin{align*}
c \beta + \sum_{(i,p)=1} c_i [ T'^{-i} ]
\end{align*}
 with $c \in W(\F_p)$ and $c_i \in W(k)$ with $c_i \to 0$ as $i \to \infty$. Furthermore, the following hold:
\begin{enumerate}
\item if $v(x_j)>0$ for all $j$, then $x \in \wp W(K)$;
\item if $v(x_j) \geq 0$ for all $j$, then $x \equiv c \beta \pmod{\wp W(K)}$;
\item let $j$ be minimal such that $v(x_j)<0$, and assume $p \nmid v(x_j)$, then $\min_i(v(c_i)) = j$. 
\end{enumerate}
\end{proposition}
\begin{proof}
The first part follows from Proposition \ref{182} and Example \ref{183}. 
Property i and ii follow directly from the proof. For iii, one has $x \equiv c' \beta+ p^j (x_j,x_{j+1},\ldots) \pmod{\wp W(K)}$ by ii. Since $p \nmid v(x_j)<0$, one has $v(c_i)=j$ for some $i$. 
\end{proof}

We will now show that the properties in Proposition \ref{711} uniquely determine the symbol.

\begin{proposition} \label{713}
There is a unique map $\langle\ ,\ \rangle: W(K)/\wp W(K) \times \widehat{K^*} \to \Z_p$ having the properties of Proposition \ref{711}. 
\end{proposition}
\begin{proof}
Let $\langle\ ,\ \rangle$ be such a map. 

Let $x \in W(K)$ and $y \in \widehat{K^*}$. One can write $y=T^{v-1}T'$ with $T'$ of valuation $1$. By property i, one has $\langle x,y \rangle=(v-1)\langle x,T \rangle+\langle x,T' \rangle$. Hence we can restrict to the computation of $\langle x,T' \rangle$ with $T'$ a uniformizer. Modulo $\wp W(K)$, we can write $x=c \beta + \sum_{(i,p)=1} c_i [ T'^{-i} ]$ (Proposition \ref{712}). By property i one has
\begin{align*}
\langle x, T' \rangle =c\langle \beta, T' \rangle + \sum_{(i,p)=1} \langle c_i [ T'^{-i} ],T' \rangle.
\end{align*}
We claim that $\langle c_i [ T'^{-i} ],T' \rangle=0$. Let $\mathcal{B}$ be a basis of $k$ over $\F_p$. We can write $c_i=\sum_{b \in \mathcal{B}} a_b [b]$ with $a_b \in W(\F_p)$. By property i, it is enough to show $\langle [b T'^{-i}],T' \rangle=0$. This follows from property iii. 
Hence one finds by property ii
\begin{align*}
\langle x, T' \rangle = c \langle \beta, T' \rangle = c \mathrm{Tr}_{W(k)/W(\F_p)}(\beta).
\end{align*}
\end{proof}

Technically speaking, the above proof gives us an algorithm for computing the symbol $[\ ,\ )$, but it involves changing the uniformizer and hence the decomposition as in Proposition \ref{712}. This makes the computation of the symbol not very flexible. We are mostly interested in fixing an element $x \in W(K)/\wp W(K)$, fixing one representation in Proposition \ref{712} and computing $[x,\cdot)$, in order to for example compute conductors and ramification groups. In the next sections, we discuss methods for doing this.

\subsection{Classical formula for $[\ ,\ )$}

We will describe a classical formula for computing the Schmid-Witt symbol. In the next subsection, we will discuss a simplified formula. The classical formula comes from \cite[Satz 18]{WITT} and \cite{SCH1}, and a more modern treatment can be found in \cite[Proposition 3.4]{THO}. The idea is to lift the computation to characteristic $0$ and work explicitly with ghost components. Let $R$ be a ring. We denote by $g=(g^{(0)}, g^{(1)}, \ldots): W(R) \to R^{\Z_{\geq 0}}$ the ghost component map.

The ring morphism $\pi_1: W(k) \to W_1(k)=k$ induces ring morphisms
\begin{align*}
P_1=W(\pi_1) \colon W(W(k)) \to W(k) \\
(a_i)_i \mapsto (\pi_1(a_i))_i,
\end{align*}
and
\begin{align*}
P_2 \colon W(k)((T)) \to& K=k((T))\\
\sum_i a_i T^i \mapsto& \sum_i \pi_1(a_i) T^i,
\end{align*}
and
\begin{align*}
P_3=W(P_2) \colon W(W(k)((T))) \to& W(K)=W(k((T))) \\
 (\sum_j a_{ij}T^j)_i \mapsto& (\sum_j \pi_1(a_{ij})T^j)_i.
\end{align*}
Furthermore, we have the logarithmic derivative map (a group morphism)
\begin{align*}
\mathrm{dlog}: W(k)((T))^* \to& W(k)((T)) \\
f \mapsto& df/f
\end{align*}
where $df$ is the derivative of $f$ (formally as power series). Finally, we have the residue map (a group morphism)
\begin{align*} 
\mathrm{Res}: W(k)((T)) \to& W(k) \\
\sum_i a_i T^i \mapsto& a_{-1}.
\end{align*}
Let $x \in W(K)$ and $y \in K^*$. Let $X \in W(W(k)((T)))$ with $P_3X=x$ and $Y \in W(k)((T))$ with $P_2Y=y$. 
One has $\mathrm{dlog} Y, g^{(i)}X \in W(k)((T))$, and set $Z=\left( \mathrm{Res}( \mathrm{dlog}  Y \cdot g^{(i)}X )\right)_{i=0}^{\infty} \in W(k)^{\Z_{\geq 0}}$. One then considers $Z'=g^{-1}Z \in W(W(k))$ (it requires some work to show that the inverse of the ghost map lands in $W(W(k))$) and $P_1 Z' \in W(k)$. Finally one can take the trace to obtain an element of $W(\F_p)$. This construction does not depend on the choice of $X$ and $Y$. 

\begin{theorem} \label{f1}
Let $x \in W(K)$ and $y \in K^*$. Let $X \in W(W(k)((T)))$ with $P_3X=x$ and $Y \in W(k)((T))$ with $P_2Y=y$. Then one has
\begin{align*}
[x,y) = \mathrm{Tr}_{W(k)/W(\F_p)} \left( P_1 g^{-1} \left( \mathrm{Res} ( \mathrm{dlog} Y \cdot g^{(i)}X ) \right)_{i=0}^{\infty}  \right).
\end{align*}
\end{theorem}
\begin{proof}
See \cite{THO}, Proposition 3.4. 
\end{proof}

\begin{remark} \label{299}
From the above formula, it is quite easy to see that for a fixed $x \in W(K)$ there is an integer $i$ with $[x,1+T^i k[[T]])=0$, and this implies the existence of the conductor corresponding to $x$. We can phrase this in the following way. We can endow $\widehat{K^*} \cong T^{\Z_p} \times U^1$ with a different topology. On $U^1$ we take the topology coming from the valuation on $k((T))$, with open sets around $1$ of the form $1+T^j k[[T]]$. Then the symbol $[\ ,\ ): W(K)/\wp W(K) \times \widehat{K^*} \to W(\F_p)$ is continuous with this new topology. 
\end{remark}

Unfortunately, it is not obvious to see from the above formula that the properties in Proposition \ref{711} are satisfied.

From the formula above, one can deduce the following simplified formula \cite[Proposition 3.5]{THO} with $X$ and $Y$ as above:
\begin{align*}
[x,y)_n = \pi_n ( \mathrm{Tr}_{W(k)/W(\F_p)} \mathrm{Res}( g^{(n-1)} X \cdot \mathrm{dlog} Y )).
\end{align*}
This formula can then be used to compute conductors in $\Z/p^n\Z$-extensions. However, computations turn out to be quite tedious, and this formula is not a convenient formula for computing $[x,y)$ itself. Furthermore, it still involves computing one ghost component. The formula for $n=1$ is particularly easy: 
\begin{align*}
[x,y)_1= \mathrm{Tr}_{k/\F_p} \left( \mathrm{Res}(x \frac{dy}{y}) \right). 
\end{align*}
In the next section, we deduce a formula for $[\ ,\ )$ which is strikingly similar to the formula for $n=1$. Also, our formula resembles formulas for similar symbols constructed by using Kummer theory when $K$ is a finite field extension of $\Q_p$.

\subsection{A new formula for $[\ ,\ )$}

Consider the ring 
\begin{align*}
R=\left\{\sum_{i \in \Z} a_i T^i: a_i \in W(k),\ \lim_{i \to -\infty} a_i=0\right\}= \underset{\underset{i}{\leftarrow}}\lim\ W(k)/V^iW(k) ((T))
\end{align*}
of two sided power series with some convergence property. We have a residue map
\begin{align*}
\mathrm{Res}: R \to& W(k) \\
\sum_i a_i T^i \to& a_{-1}. 
\end{align*}
Let $x \in W(K)$. Let $c\beta + \sum_{(i,p)=1} c_i [  T]^{-i}$ be its unique representative modulo $\wp W(K)$ as in Proposition \ref{712}. We define 
\begin{align*}
\tilde{}: W(K)/\wp W(K) \to& R \\
c\beta + \sum_{(i,p)=1} c_i [  T]^{-i} \pmod{\wp W(K)} \mapsto& c\beta + \sum_{(i,p)=1} c_i T^{-i}.
\end{align*}
Any element $y \in \widehat{K^*} \cong T^{\Z_p} \times \left( 1+ Tk[[T]] \right)$ can uniquely be written as 
\begin{align*}
y=T^e \cdot \prod_{(i,p)=1}^{\infty}  \prod_{j=0}^{\infty} (1- a_{ij} T^i)^{p^j}
\end{align*}
with $e \in \Z_p$ and $a_{ij} \in k$. 
We first set
\begin{align*}
\tilde{ }: \widehat{K^*} \cong T^{\Z_p} \times \left( 1+ Tk[[T]] \right) \to& T^{\Z_p} \times \left( 1+ TW(k)[[T]] \right) \\
T^e \cdot \prod_{(i,p)=1}^{\infty}  \prod_{j=0}^{\infty}  (1- a_{ij} T^i)^{p^j} \mapsto& T^e \cdot \prod_{(i,p)=1}^{\infty}  \prod_{j=0}^{\infty}  (1- [a_{ij}]T^i)^{p^j}.
\end{align*}
Note that for $y \in K^*$ one has $P_2\tilde{y}=y$. 
Furthermore, we define the group morphism
\begin{align*}
\mathrm{dlog}: T^{\Z_p} \times \left( 1+ TW(k)[[T]] \right) \to& W(k)((T)) \\
T^e \cdot f \mapsto& \frac{e}{T} + \frac{df}{f}
\end{align*}
where $df$ is the formal derivative of $f$. 

\begin{theorem} \label{100}
Let $x \in W(K)$ and $y \in \widehat{K^*}$. Then one has 
\begin{align*}
[x,y) = \mathrm{Tr}_{W(k)/W(\F_p)} \left( \mathrm{Res}( \tilde{x} \cdot \mathrm{dlog}\tilde{y}) \right).
\end{align*}
Equivalently, let $x \equiv c\beta + \sum_{(i,p)=1} c_i [ T]^{-i} \pmod{\wp W(K)}$ as in Proposition \ref{712}, and $y=T^e \cdot \prod_{(i,p)=1}^{\infty}  \prod_{j=0}^{\infty}  (1- a_{ij} T^i)^{p^j} \in \widehat{K^*}$ with $a_{ij} \in k$ and $d \in \Z_p$. Then one has:
\begin{align*}
[x,y) = ce\mathrm{Tr}_{W(k)/W(\F_p)}(\beta)- \sum_{j=0}^{\infty} p^j  \mathrm{Tr}_{W(k)/W(\F_p)} \left(\sum_{(i,p)=1} c_i \sum_{l|i} l [a_{lj}]^{i/l} \right) \in W(\F_p). 
\end{align*}
\end{theorem}
\begin{proof}
We show that the formula agrees with Theorem \ref{f1}. By Remark \ref{299} and the bilinearity one finds
\begin{align*}
[x,T^e \cdot \prod_{(i,p)=1}^{\infty}  \prod_{j=0}^{\infty}  (1- a_{ij} T^i)^{p^j} ) = e[x,T) +  \sum_{(i,p)=1}^{\infty}  \sum_{j=0}^{\infty} p^j[x,1-a_{ij}T^i),
\end{align*}
and a similar decomposition holds for our proposed formula.
By Proposition \ref{712} and bilinearlity, one can restrict to the following three cases: $(x,y)=(x,T)$,  $(x,y)=(\beta,1-a T^i)$ and $(x,y)=([b T^{-l}],1-a T^i)$ with $a, b \in k$, $p \nmid li$. 
By Property i, ii, iii from Proposition \ref{711} with $x \equiv c\beta + \sum_{(i,p)=1} c_i \{ T\}^{-i} \pmod{\wp W(K)}$,  one finds:
\begin{align*} 
[x,T) = [c \beta,T) = \mathrm{Tr}_{W(k)/W(\F_p)}(c \beta). 
\end{align*} 
By Proposition \ref{711} ii one has
\begin{align*}
[\beta,1-aT^i)=0.
\end{align*}
Our formula gives the same in the first two cases. 
Finally, we consider the case $(x,y)=([b T^{-l}],1-a T^i)$. We follow Theorem \ref{f1} and pick the following lifts. Set $Y=1-[a] T^i=\tilde{y} \in W(k)((T))$. Set $X=[[b]T^{-l}] \in W(W(k)((T)))$. One has 
\begin{align*}
\mathrm{dlog}Y = \mathrm{dlog}\tilde{y} = -i \sum_{k \geq 1} [a]^k T^{ik-1}
\end{align*} 
and $g(X)= ( [b]^{p^s} T^{-lp^s})_{s=0}^{\infty}$. One then has
\begin{align*}
\left( \mathrm{Res} ( \mathrm{dlog} Y \cdot g^{(s)}X ) \right)_{s=0}^{\infty}=&   \left( \mathrm{Res} (  -i \sum_{k \geq 1} [a]^k T^{ik-1} \cdot  [b]^{p^s} T^{-lp^s} )   \right)_{s=0}^{\infty} \\
=& \left\{ \begin{array}{cc} -i ( [a^{l/i} b ]^{p^s})_{s=0}^{\infty}=-ig([[a^{l/i} b ]]) &\textrm{if } i\mid l \\
0 & \textrm{if } i \nmid l. \end{array} \right.
\end{align*}
One then finds:
\begin{align*}
[x,y) = &\left\{ \begin{array}{cc} \mathrm{Tr}_{W(k)/W(\F_p)} \left( -i[a^{l/i} b] \right) &\textrm{if } i\mid l  \\ 0 & \textrm{if } i \nmid l  \end{array}  \right. \\
=&  \mathrm{Tr}_{W(k)/W(\F_p)} \left( \mathrm{Res}( [b] T^{-l} \cdot -i \sum_{k \geq 1} [a]^k T^{ik-1} ) \right) \\  
=& \mathrm{Tr}_{W(k)/W(\F_p)} \left( \mathrm{Res}( \tilde{x} \cdot \mathrm{dlog}\tilde{y}) \right).
\end{align*}
The explicit formula follows from an explicit computation, where
\begin{align*}
\mathrm{dlog} \tilde{y} = e/T  - \sum_{j=0}^{\infty} p^j \sum_{(i,p)=1}^{\infty}  i \sum_{k \geq 1} [a_{ij}]^k T^{ik-1}.
\end{align*}
and
\begin{align*}
\tilde{x}=c \beta+ \sum_{(i,p)=1} c_i T^{-i}.
\end{align*}
The formula then follows from 
\begin{align*}
\mathrm{Res}(\tilde{x} \cdot \mathrm{dlog} \tilde{y}) = ce \beta - \sum_{j=0}^{\infty} p^j  \sum_{(i,p)=1} c_i \sum_{l|i} l [a_{lj}]^{i/l}.
\end{align*}

\end{proof}

\begin{remark}
The formula in Theorem \ref{100} can be used to compute $[x,y)_n$ given $x,y,n$ in polynomial time (if enough precision of $x,y$ are given). We leave the details to the readers.

When $K$ is a finite extension of $\Q_p$ with residue field $k$ containing a primitive $p^n$th root of unity $\zeta$, there is a similar symbol $(\ ,\ )_{p^n}: K^* \times K^* \to \langle \zeta \rangle$ obtained from combining local class field theory and Kummer theory. There exists a formula for this symbol which is quite similar to our formula for $[\ ,\ )$. In the notation of \cite[Chapter VII]{FESE} one has for $p>2$
\begin{align*}
(x,y)_{p^n} = \zeta^{\mathrm{Tr}_{W(k)/W(\F_p)}(\mathrm{Res}(f_{x,y}g))}
\end{align*}
where $f_{x,y}$ is a power series in $W(k)((X))$ depending on $x$ and $y$ and $g$ is a power series in $W(k)((X))$ depending on $K$. For $p=2$, a similar formula exists. 

Furthermore, there is a deterministic polynomial time algorithm which on input a local number field $K$, an integer $p^n$ and $x,y, \in K^*$ checks if $K$ contains a primitive $p^n$th root of unity and if so, computes the symbol $(x,y)_{p^n}$ \cite{BOUW}.
\end{remark}

\begin{remark}
Let $y \in K^*$ be a polynomial with constant term $1$. Factor $y=\prod_{i=1}^n (1-\alpha_i T) \in \overline{k}[T]$. One can lift $y$ by taking $Y=\prod_{i=1}^n (1-[\alpha_i] T) \in W(k)[T]$. Let $x \in W(K)$ with $x \equiv \sum_{(i,p)=1} c_i [T]^{-i}  \pmod{\wp W(K)}$. 
Then one finds, by following the above proof,
\begin{align*}
[x,y)= -\mathrm{Tr}_{W(k)/W(\F_p)} \left(\sum_{i=1}^n   \tilde{x}([\alpha_i]^{-1}) \right).
\end{align*}
Hence one can view the symbol as an evaluation symbol. This formula simplifies to $-\mathrm{Tr}_{W(k(\alpha_j))/W(\F_p)}  \left( \tilde{x}[\alpha_j]^{-1} \right)$ if $y$ is irreducible. 
\end{remark}

\begin{remark} \label{900}
Let us give another interpretation of the symbol $[\ ,\ )$. Consider the map $\psi: \widehat{K^*} \to G_p$ from class field theory. 
One has $\widehat{K^*} \cong T^{\Z_p} \times U^1$.
One has $G_p \cong \Hom(W(K)/\wp W(K),W(\F_p))$. Note that  
\begin{align*}
\Hom(W(K)/\wp W(K),W(\F_p)) \cong& \Hom(W(\F_p),W(\F_p)) \times \prod_{(i,p)=1} \Hom(W(k),W(\F_p)) \\
\chi \mapsto& \left( (a \mapsto \chi(a \beta) ), ( b \mapsto \chi(b[T^{-i}]))_i \right). 
\end{align*}
The natural map
\begin{align*}
W(k) \to& \Hom(W(k),W(\F_p)) \\
a \mapsto& (x \mapsto \mathrm{Tr}_{W(k)/W(\F_p)}(xa)).
\end{align*}
is an isomorphism because the map $\mathrm{Tr}_{k/\F_p}$ and $\mathrm{Tr}_{W(k)/W(\F_p)}$ are non-degenerate. Combining givues us an isomorphism
\begin{align*}
\Hom(W(K)/\wp W(K),W(\F_p)) \cong  W(\F_p) \times \prod_{(i,p)=1} W(k). 
\end{align*}
Class field theory gives $\widehat{K^*} \cong \Hom(W(K)/\wp W(K),W(\F_p))$. Using our explicit formule in Theorem \ref{100},  we obtain an isomorphism
\begin{align*}
T^{\Z_p} \times U^1 \cong& W(\F_p) \times \prod_{(i,p)=1} W(k), \\
T^e \cdot \prod_{(i,p)=1}^{\infty}   \prod_{j=0}^{\infty}  (1- a_{ij} T^i)^{p^j} \mapsto& \left(e\mathrm{Tr}_{W(k)/W(\F_p)}(\beta), \left( \sum_{j=0}^{\infty} p^j \sum_{l|i} l [a_{lj}]^{i/l} \right)_i \right).
\end{align*}
This map induces isomorphisms $T^{\Z_p} \cong W(\F_p)$ and $U^1 \cong \prod_{(i,p)=1} W(k)$. The second map is known, see for example \cite[Proposition 1.10]{HESS}. We can use these map to define $[\ ,\ )$. To check that the properties of Proposition \ref{711} hold, it is important to show invariance when a different uniformizer is chosen: how the decomposition in $W(K)/\wp W(K)$ as in Proposition \ref{712} and the map $U^1 \cong \prod_{(i,p)=1} W(k)$ change. 
\end{remark}

\begin{remark}
Generalizations of $[\ ,\ )$ exist, see for example \cite[6.4]{KAT}. Formulas here use ghost components explicitly and it would be interesting to see if our approach can be generalized.
\end{remark}

\subsection{Conductors, discriminants, and ramification groups}

\subsubsection{Ramification groups}

For $r \in \Z_{\geq -1}$, let $G_p^r$ denote the $r$th upper ramification group of $G_p$ (see \cite[Chapter IV]{SE1} for more details). Our goal is to compute the image $H^r$ of $G_p^r$ under the isomorphism 
$$\tau: G_p \to H=\Hom(W(K)/\wp W(K),W(\F_p))$$
from Theorem \ref{15}. 
Let $\psi: \widehat{K^*} \to G_p$ be the homeomorphism from class field theory. 
One can compute the ramification groups from $\psi$ as follows (\cite[Chaper XV, Theorem 2]{SE1}). Set $U^i=1+T^i k[[T]] \subset \widehat{K^*}$ for $i \in \Z_{\geq 1}$. 
One has $H^{-1}=\tau \circ \psi(\widehat{K^*})=H$ (decomposition group) and
\begin{align*}
H^0 = \tau \circ \psi (k^* U^1 ) = \{ \chi \in H: \chi(\beta)=0\}.
\end{align*}
corresponding to the unramified extension (inertia group). 
For $r \in \Z_{\geq 1}$,  one has
\begin{align*}
H^r = \tau \circ \psi(U^r).
\end{align*}
Note that $k^*U^1 = U^1$ as subgroups of $\widehat{K^*}$. It follows that 
$H^0=H^1$ (there is only wild ramification). Since we know the map $\tau \circ \psi$ explicitly due to the Schmid-Witt symbol, we can compute the ramification groups explicitly. 
Recall that we have explicit isomorphisms (Remark \ref{900})
\begin{align*}
\phi_1: U^1=1+Tk[[T]] \to& \prod_{(i,p)=1} W(k) \\
\prod_{(i,p)=1}^{\infty}   \prod_{j=0}^{\infty}  (1- a_{ij} T^i)^{p^j} \mapsto& \left( \sum_{j=0}^{\infty} p^j \sum_{l|i} l [a_{lj}]^{i/l} \right)_i
\end{align*}
and 
\begin{align*}
\phi_2: \prod_{(i,p)=1} W(k) \to& H^0 \\
(a_i)_i \mapsto& \left( c \beta+ \sum_i c_i [T]^{-i} \mapsto \mathrm{Tr}_{W(k)/W(\F_p)} (\sum_i a_i c_i) \right) .
\end{align*}
such that $\phi_2 \circ \phi_1=\psi$. One has for $r \in \Z_{\geq 1}$
\begin{align*}
U^r = \left\{\prod_{(i,p)=1}^{\infty}   \prod_{j=0}^{\infty}  (1- a_{ij} T^i)^{p^j}: ip^j< r \implies a_{ij}=0\right\} \subseteq U^1.
\end{align*}

\begin{theorem} \label{up}
For $r \in \Z_{\geq 1}$, one has 
\begin{align*}
\phi_1(U^r)=(p^{ \lceil  \log_p \left(\frac{r}{i}\right)  \rceil } W(k))_i
\end{align*}
and
\begin{align*}
H^r = \{ \chi \in H^0, v(\chi([T]^{-i})) \geq \lceil \log_p (r/i) \rceil, (i,p)=1 \}.
\end{align*}
\end{theorem}
\begin{proof}
The condition $ip^j < r$ is equivalent to $j < \log_p (r/i)$. It follows that 
$$\phi_1(U^r) \subseteq (p^{ \lceil  \log_p \left(\frac{r}{i}\right)  \rceil } W(k))_i.$$
To prove the other inclusion, we 
fix $i$ with $(i,p)=1$ and take $j=\lceil  \log_p \left(\frac{r}{i}\right)  \rceil$. One checks that $(1-a_{ij}T^i)^{p^j} \in U^r$ and  
$$\phi_1((1-a_{ij}T^i)^{p^j}) = (b_k)_{(k,p)=1}, \ b_i = p^j i [a_{ij}], \ b_k = 0 \ (k\not =i).$$
By linearity and continuity, this proves the other inclusion.
The second result follows by applying $\phi_2$, since $\mathrm{Tr}_{W(k)/W(\F_p)}$ is non-degenerate. 
\end{proof}

See \cite[Theorem 1.1]{THO} for a different description of the upper ramification groups. 

If $L/K$ is Galois with groups $H_1$ and $L'/K$ is an intermediate Galois extension with group $H_1/H_2$, then one has 
\begin{align*}
(H_1/H_2)^r=H_1^r H_2/H_2. 
\end{align*}
Hence this allows us to compute upper ramification groups of intermediate extensions. 

\subsubsection{Conductors and discriminants}

Let $L/K$ be a finite abelian extension. By class field theory, we have a surjective map $\psi_L: K^* \to \Gal(L/K)$ (which is the restriction of the global Artin map). The conductor of $L/K$ is $\mathfrak{f}(L/K)=\pa^i$ with $i$ minimal such that $U^i \subseteq \mathrm{ker}(\psi_L)$. Here we set $U^0=k[[T]]^*=k^* U^1$. If $L=L_1L_2$, then one has 
\begin{align*}
\mathfrak{f}(L/K) = \lcm( \mathfrak{f}(L_1/K), \mathfrak{f}(L_2/K)). 
\end{align*}
Hence to study the conductor of extensions of exponent dividing $n$, it is enough to study the conductor of an extension of the form $K(\wp^{-1}x)/K$ with $x \in W_n(K)$. 
Let $x \in W(K)$. Our goal is to study the ramification in $K(\wp^{-1}x)/K$. We have a natural `valuation' $v$ on $W(k)$ with $v(x_0,x_1,\ldots)=i$ where $i$ is minimal with $x_i \neq 0$. We set $K_n=K(\wp^{-1}\pi_n(x))/K$. We first study the extensions $K_n/K$.

\begin{lemma}
Let $x \equiv  c \beta + \sum_{(i,p)=1} c_i [ T]^{-i} \pmod{\wp W(K)}$ in reduced form. Let 
\begin{align*} 
n_c=\min\{v(c_i):i \},\ n_0= \min\{n_c,v(c)\}.
\end{align*}
Then $K_{n_0}/K_0$ is a trivial extension, $K_{n_c}/K$ is unramified cyclic of degree $p^{n_c-n_0}$ and $K_n/K$ for $n>n_c$ is cyclic of degree $p^{n-n_0}$ with ramification index $p^{n-n_c}$. The extension $K_{n_c}/K_{n_0}$ is a constant field extension, and $K_{n}/K_{n_c}$ for $n \geq n_c$ is totally ramified.
\end{lemma}
\begin{proof}
One has $\pi_{n_0}(x) \in \wp W_n(K)$, and hence $K_{n_0}/K$ is trivial. It is easy to see that $K_{n_c}/K_{n_0}$ just extends the residue field, and after that, it becomes totally ramified.
\end{proof}
We call $x \in W(K)$ primitive if $n_0$ as in the above theorem is $0$. In general one can find $y \in W(K)$ with $p^{n_0} \overline{y} =\overline{x} \in W(K)/\wp W(K)$ and such an $y$ is primitive and one has $K(\wp^{-1}\pi_n(y))=K_{n+n_0}$. 
We set $\mathfrak{f}_n=\mathfrak{f}(K_n/K)$. In other words, it is equal to $\pa^i$ with $i$ minimal such that for all $y \in U^i$ one has $[\pi_n(x),y)_n=0$, equivalently, $[x,y) \in p^n W(\F_p)$ for all $y \in U^i$. 
Let us use the above formula to compute conductors (see also \cite[Corollary 5.1]{THO} or \cite{SCH1}). 

\begin{proposition} \label{715}
Let $x \in W(K)$ with $x \equiv c \beta + \sum_{(i,p)=1} c_i [ T]^{-i} \pmod{\wp W(K)}$ with $n_0$ as above. Let $y \in W(\overline{K})$ be such that $\wp y=x$. Consider the isomorphism
\begin{align*}
H'=\Gal(K(\wp^{-1}x)/K) \to& \Hom(\langle x \rangle, W(\F_p)).
\end{align*}
Let $r \in \Z_{\geq 1}$ and set
\begin{align*}
b= \left\{ \begin{array}{cc} 
\min\{ v(c_i)+ \lceil  \log_p \left(\frac{r}{i}\right) \rceil: (i,p)=1 \}  & \textrm{if } r \geq 1 \\
\min\{ v(c_i) : (i,p)=1 \} & \textrm{if } r=0.
\end{array} \right.
\end{align*}
Under this isomorphism $H'^r$ corresponds to 
\begin{align*}
\left\{ \tau: v(\tau(x)) \geq b \right\} \subseteq \Hom(\langle x \rangle, W(\F_p)).
\end{align*}
Let $n \in \Z_{\geq 1}$. 
One has $\mathfrak{f}_n=\mathfrak{f}(K(\wp^{-1} \pi_n(x))/K)=\pa^{u_n}$ where
\begin{align*}
u_n = \left\{  \begin{array}{cc} 1+ \max\{ ip^{n-v(c_i)-1}  : (i,p)=1: v(c_i)<n\} & \mathrm{if\ }\exists i: v(c_{i})<n \\
0 & \mathrm{otherwise}. \end{array} \right.
\end{align*}
\end{proposition}
\begin{proof}
The first statement follows from the computation of $H'^r$ in Theorem \ref{up}. 
The exponent of the conductor of $K_n/K$ is $0$ if for all $i$ one has $v(c_i) \geq n$. Otherwise, it is the smallest integer $r \geq 1$ such that 
\begin{align*}
\min_{} \left\{ v(c_i)+ \lceil  \log_p \left(\frac{r}{i}  \right) \rceil: (i,p)=1 \right\} \geq n.
\end{align*} 
That is, the minimal $r \geq 1$ such that for all $i$ one has 
\begin{align*}
v(c_i)+ \lceil  \log_p \left(\frac{r}{i}\right)  \rceil \geq n,
\end{align*}
equivalently $\lceil  \log_p \left(\frac{r}{i}\right)  \rceil \geq n-v(c_i)$.
Fix an integer $i$. The above condition is automatically satisfied if $v(c_i)\geq n$. We now assume that $v(c_i)<n$. Note that $\log_p(r/i)$ is a strictly increasing function of $r$. One has $\log_p(r/i) =n- v(c_i)-1$ at the integer $r=ip^{n-v(c_i)-1}$. Hence one finds $ \lceil  \log_p \left(\frac{r}{i}\right)  \rceil \geq n-v(c_i)$ if and only if 
\begin{align*}
r \geq i p^{n-v(c_i)-1}+1
\end{align*}
and the results follow.

The result for $n_0>0$ follows in a similar way, by replacing $x$ by `$x/p^{n_0}$'. 
\end{proof}

As expected, the conductor is $\mathfrak{f}_n=\pa^0$ for $n \leq n_c$. For $n>n_c$ we find 
\begin{align*}
u_n \geq 1 + p^{n-n_c-1} \geq 2
\end{align*}
That $u_n\geq 2$, means that the ramification, as expected, is wild. 

Let us now compute the discriminant $\Delta_n$, a power of $\pa$, of $K_n/K_0$. See \cite[Chapter III]{SE1} for a definition involving the norm of the different. 

\begin{proposition} \label{144}
Let $n \in \Z_{\geq 0}$ and let $n_0$ be as above. Then 
\begin{align*}
\Delta_n = \left\{ \begin{array}{cc} \prod_{i=n_0}^n \mathfrak{f}_i^{\varphi(p^{i-n_0})}  & \textrm{if }n > n_0 \\ \pa^0 & \textrm{if }n \leq n_0 \end{array} \right.
\end{align*}
with $\varphi(p^s)=p^{s-1}(p-1)$. 
\end{proposition}
\begin{proof}
The F\"uhrerdiskriminanten\-produktformel \cite{SE1} gives
\begin{align*}
\Delta_{n} =\prod_{\chi \in \Hom(\Gal(K_n/K),\C^*)} \mathfrak{f}(K_n^{\mathrm{ker}(\chi)}/K) =  \prod_{i=n_0}^n \mathfrak{f}_i^{\varphi(p^{i-n_0})} .
\end{align*}
\end{proof}

\section{$\Z_p$-extensions of function fields} \label{s5}

Let $k$ be a finite field. Let $K=K_0$ be a function field over $k$ (a finitely generated field extension of $k$ of transcendence degree $1$) with full constant field $k$. Let $x=(x_0,x_1,\ldots) \in W(K)$. This Witt vector defines a field extension $K_{\infty}/K$. For simplicity, we assume that $x_0 \not \in \wp W(K)$. Set $K_i=K(\wp^{-1} \pi_i(x))$. One then has a chain of field $K=K_0 \subseteq K_1 \subseteq K_2 \subseteq \ldots \subset K_{\infty}$ with $\Gal(K_n/K) \cong \Z/p^n\Z$ and $\Gal(K_{\infty}/K) \cong \Z_p$. 

\subsection{Genus formula}

Let $\pa$ be a place of $K$ with residue field $k_{\pa}$ and uniformizer $\pi_{\pa}$. Then, locally, this extension is given by $x=(x_0,x_1,\ldots) \in W(K_{\pa})$ where $K_{\pa} \cong k_{\pa}((\pi_{\pa}))$ is the completion at $\pa$ (by the Cohen structure theorem). Let $\alpha_{\pa} \in k_{\pa}$ with $\mathrm{Tr}_{k_{\pa}/\F_p}(\alpha_{\pa}) \neq 0$. Set $\beta_{\pa}=[\alpha_{\pa}] \in W(k_{\pa})$. One has 
$$x \equiv c_{\pa} \beta_{\pa} + \sum_{(i,p)=1} c_{\pa,i} [\pi_{\pa}]^{-i} \pmod{\wp W(K_{\pa})}$$
with $c_{\pa} \in W(\F_p)$ and $c_{\pa,i} \in W(k_{\pa})$ and $c_{\pa,i} \to 0$ as $i \to -\infty$ (Proposition \ref{712}). We set 
\begin{align*}
n_{\pa,c}=\min \{v(c_{\pa,i}):i \},\ n_{\pa,0}= \min\{n_{\pa,c},v(c_{\pa})\}.
\end{align*}

Proposition \ref{715} then shows that the conductor at $\pa$ of $K_n/K$ is equal to $\mathfrak{f}_{\pa,n} =\pa^{u_{\pa,n}}$ with
\begin{align*}
u_{\pa,n} = \left\{  \begin{array}{cc} 1+ \max\{ ip^{n-v(c_{\pa,i})-1}  : (i,p)=1: v(c_{\pa,i})<n\} & \mathrm{if\ }\exists i: v(c_{\pa,i})<n \\
0 & \mathrm{otherwise}. \end{array} \right.
\end{align*}
The conductor of $K_n/K$ is 
\begin{align*}
\mathfrak{f}_n=\mathfrak{f}(K_n/K) = \prod_{\pa} \mathfrak{f}_{\pa,n},
\end{align*}
which is a finite product. 

Let $g_n$ be the genus of $K_n$. We let $n_c$ be maximal such that $K_{n_c}/K$ is a constant field extension. For all $\pa$ this implies $\mathfrak{f}_{\pa,i}=\pa^0$ for $i \leq n_c$.

\begin{proposition} \label{444}
Let $\Delta_n$ be the discriminant of $K_n/K$. One has:
\begin{align*}
\Delta_n = \prod_{\pa} \prod_{i=0}^n \mathfrak{f}_{\pa,i}^{ \varphi(p^{i } ) }  
\end{align*}
\end{proposition}
\begin{proof}
A prime $\pa$ of $K$ has $p^{\min(n,n_{\pa,0})}$ extensions in $K_n$. We then use \cite[Chapter III]{SE1} and Proposition \ref{144} to compute the discriminant, since $\mathfrak{f}_{\pa,i}=\pa^0$ for $i \leq n_{\pa,0}$:
\begin{align*}
\Delta_n = \prod_{\pa} \Delta_{\pa,n}^{p^{\min(n,n_{\pa,0})}} =&  \prod_{\pa:\ n >n_{\pa,0}} \prod_{i=n_{\pa,0}+1}^n \mathfrak{f}_{\pa,i}^{ \varphi(p^{i-n_{\pa,0}}) \cdot p^{n_{\pa,0}}} \\
=& \prod_{\pa:\ n >n_{\pa,0}} \prod_{i=n_{\pa,0}+1}^n \mathfrak{f}_{\pa,i}^{ \varphi(p^{i } ) } \\
=&  \prod_{\pa} \prod_{i=0}^n \mathfrak{f}_{\pa,i}^{ \varphi(p^{i } ) }.
\end{align*}
\end{proof}
The following result was already obtained in \cite{SCH1}. 

\begin{theorem} \label{831}
One has:
\begin{align*}
p^{\min\{n_c,n \}}(2 g_n - 2) =& p^{n} ( 2 g_0 - 2 ) + \sum_{\pa} [k_{\pa}:k] \sum_{i=0}^n \varphi(p^i) u_{\pa,i}.
\end{align*}
\end{theorem}
\begin{proof}
Let $\mathfrak{D}_n$ be the different of $K_n/K$. 
Let $k_n$ be the full constant field of $K_n$. Then the Riemann-Hurwitz formula gives (\cite{ST} or \cite[Corollary 9.4.3]{SALV}):
\begin{align*}
[k_n:k](2 g_n - 2) =& p^{n} ( 2 g_0 - 2 ) + \deg_k \mathfrak{D}_n.
\end{align*}
One has $[k_n:k]=p^{\min \{n_c,n \}}$ and $\deg_k \mathfrak{D}_n=\deg_k \Delta_n$ as in Proposition \ref{444}. The result follows.
\end{proof}

\subsection{Genus lower bounds}

Let $n_u$ be maximal such that $K_{n_u}/K$ is unramified. One has $n_u \geq n_c$. 

\begin{corollary} \label{15b}
Let $n \geq n_u$. The following statements hold:

\begin{enumerate}
	\item
\begin{align*}
p^{n_c} (2g_n-2)  \geq p^n (2g_0-2) +p^n-p^{n_u}+ p^{n_u} \frac{p^{2(n-n_u)}-1}{p+1} .
\end{align*}

\item   
\begin{align*}
\liminf_{n \to \infty} \frac{g_n}{p^{2n}} \geq \frac{1}{2p^{n_u+n_c}(p+1)}.
\end{align*}

\item  For any $\epsilon>0$,  there is an integer $m$ such that for all $n \geq m$ one has
\begin{align*}
g_n \geq \frac{p^{2n-n_u-n_c}}{2(p+1)+\epsilon} \geq \frac{p^{2n-n_u-n_c-1}}{3+\epsilon}.
\end{align*}
\end{enumerate}
\end{corollary}
\begin{proof}
We try to make the genus as small as possible in the genus formula. The smallest genus is obtained if only one prime $\pa$ is ramified with $[k_{\pa}:k]=1$, such that $u_{\pa,n}=1+p^{n-n_u-1}$ for $n>n_u$, and $u_{\pa,n}=0$ for $n \leq n_u$ (see Proposition \ref{715}).
One finds for $n \geq n_u$ by Theorem \ref{831}:
\begin{align*}
p^{n_c} (2g_n-2)  \geq&  p^n (2g_0-2) + \sum_{i=n_u+1}^n \varphi(p^i) \left( 1 + p^{i-n_u-1} \right) \\
=& p^n (2g_0-2) +p^n-p^{n_u}+ p^{n_u} \frac{p^{2(n-n_u)}-1}{p+1}.  
\end{align*}
The first part is proved.
The second and third part follow by looking at the last term of the formula from the first part.
\end{proof}

\begin{remark} \label{goa}
The bounds in Corollary \ref{15b} are often sharp when the $p$-part of the class group of $K$ is $0$. In particular, the bounds are sharp when $K=k(X)$, the projective line. We will give explicit examples later.

Gold and Kisilevsky in  \cite[Theorem 1]{GOL} state that for large $n$, if $n_c=0$:
\begin{align*}
g_n \geq \frac{p^{2(n-n_u)-1}}{3}.
\end{align*}
This result contains a small error which makes the result incorrect for $p=2$, $g_0=0$ (one really needs the $\epsilon$ in that case, see Proposition \ref{hulu2}). Secondly, in their proof they reduce to the case $n_u=0$, but they forget that if $n_u>0$, then more primes must ramify and hence the genus will grow faster.

Assume from now on that $n_u=0$. In fact Gold and Kisilevsky prove in an intermediate step
\begin{align*}
\liminf_{n \to \infty} \frac{g_n}{p^{2n}} \geq \frac{p-1}{2p^2}.
\end{align*}
Our result actually gives the tight bound
\begin{align*}
\liminf_{n \to \infty} \frac{g_n}{p^{2n}} \geq \frac{1}{2(p+1)}.
\end{align*}
Li and Zhao in \cite{LIC} construct a $\Z_p$-tower with the property
\begin{align*}
\lim_{n \to \infty} \frac{g_n}{p^{2n}} = \frac{1}{2(p+1)}.
\end{align*}
Li and Zhao furthermore write ``It would be interesting to determine if the bound of Gold and Kisilevsky is
the best and find some $\Z_p$-extension which realizes it.'' 
Our results show that their tower actually attains our limit, and there does not exist a $\Z_p$-extension reaching 
the lower bound of Gold and Kisilevsky.
\end{remark}

\subsection{Stable genus}

For future arithmetic applications on the stable property of zeta functions in the spirit of Iwasawa theory \cite{WAN7}, we are interested in 
classifying the cases when $g_n$ for large enough $n$ stabilizes. The following result gives a complete answer. 

\begin{proposition} \label{blaat}
 Let $K_{\infty}/K$ be a $\Z_p$-extension. Assume that $n_c=0$. Then the following are equivalent:
\begin{enumerate}
\item The extension $K_{\infty}/K$ is ramified at only finitely many places and for all $\pa$ the set 
\begin{align*}
\{i p^{-v(c_{\pa,i})}: (i,p)=1 \}
\end{align*}
has a maximum. 
\item The extension $K_{\infty}/K$ is ramified at only finitely many places  and for each $\pa$ which ramifies there are $a_{\pa} \in \Q_{>0}$ and $n_{\pa} \in \Z_{\geq 0}$ such that for $n \geq n_{\pa}$ one has
\begin{align*}
u_{\pa,n}=1+ a_{\pa} p^{n}.
\end{align*}
\item There are $a,b,c \in \Q$, $m \in \Z_{\geq 0}$ such that for $n \geq m$ one has
\begin{align*}
g_n=a p^{2n}+b p^n+c.
\end{align*}
\end{enumerate}
\end{proposition}
\begin{proof}
i $\iff$ ii: This follows directly from Theorem \ref{831}. Assume that the maximum at $\pa$ is obtained at $i=i_{\pa}$. The formula gives for $n \geq  v(c_{\pa,i_{\pa}}) + 1 =n_{\pa}$:
\begin{align*}
u_{\pa,n} = 1+ \left( i_{\pa} p^{-v(c_{\pa,i_\pa})-1} \right) p^n.
\end{align*}
i $\implies$ iii: Let $m=\max\{ v(c_{\pa,i_{\pa}})+1: \pa \}$. Set 
$$d_1=\sum_{\pa} [k_{\pa}:k] \sum_{j=0}^{m-1} \varphi(p^j) {u}_{\pa, j}, \ d_2=-p^{m-1}\sum_{\pa} [k_{\pa}:k], \ r_{\pa}=i_{\pa} p^{-v(c_{\pa,i_\pa})}.$$
By Corollary \ref{15b}, we deduce that for $n>m$, 
\begin{align*}
2 g_n - 2 = & p^n ( 2 g_0 - 2 ) + d_1 +\sum_{\pa} [k_{\pa}:k] \sum_{j=m}^{n} \varphi(p^j) u_{\pa,j} \\
=& p^n ( 2 g_0 - 2 ) +d_1+ \sum_{\pa} [k_{\pa}:k] \sum_{j=m}^{n} \varphi(p^j)  (1+r_{\pa} p^{j-1} ) \\
=& \frac{p^{2n}-p^{2(m-1)}}{p+1} \left( \sum_{\pa} [k_{\pa}:k] r_{\pa} \right) + p^n \left( 2 g_0 - 2 + \sum_{\pa } [k_{\pa}:k] \right)  +\left( d_1+d_2 \right) \\
=& p^{2n}a' + p^n b'+c'.
\end{align*}

iii $\implies$ i: One easily obtains
\begin{align*}
\lim_{n \to \infty} \frac{2g_n-2}{p^{2n}} = \lim_{n \to \infty} \frac{ \sum_{\pa} [k_{\pa}:k] \max \{ ip^{-v(c_{\pa,i})}: v(c_{\pa,i})<n\} }{p+1}
\end{align*}
From this expression, one directly sees that only finitely many primes can ramify, and 
for each $\pa$ the limit $\lim_{n \to \infty}  \max \{ ip^{-v(c_{\pa,i})}: v(c_{\pa,i})<n\}$ exists. 
The genus stability further implies that for each $\pa$, $\max_{(i,p)=1} \{ ip^{-v(c_{\pa,i})}\}$ exists. In fact, the maximum is attained at a unique finite level, 
since these rational numbers (if non-zero) are distinct. 
\end{proof}

\begin{definition}
A $\Z_p$-tower $K_{\infty}/K$ is called geometric if $n_c=0$. A geometric $\Z_p$-tower $K_{\infty}/K$ is called \emph{genus-stable} if one of the equivalent conditions of Proposition \ref{blaat} is satisfied.
\end{definition}

\begin{remark} \label{917}
In future research, we would like to find relations between the zeros of $L$-functions of $\Z_p$-towers. We want to understand how the valuations of the zeros of the $L$-functions change in such a tower. The remarkable stable  relations for L-functions as in \cite{WAN7} can only exist if the genus in the tower behaves in a nice way, that is, if the corresponding geometric $\Z_p$-tower is genus-stable. 
\end{remark}

\subsection{Example: the projective line}

Let $K=k(X)$ be the function field of the projective line where $k$ is a finite field. Set $K'=\overline{k}(X)$ where $\overline{k}$ is an algebraic closure of $k$. We will study $\Z_p$-towers over $K$. For $x \in \overline{k}$ we set $\pi_x=X-x \in K'$ and we set $\pi_{\infty}=1/X \in K'$. Let $\alpha \in k$ with $\mathrm{Tr}_{k/\F_p}(\alpha) \neq 0$. Set $\beta=[\alpha]$. Analagous to Example \ref{183},  one can prove:

\begin{lemma} \label{hulu}
Every $a \in W(K)$ is equivalent modulo $\wp W(K)$ to a unique vector $b \in W(K)$ of the form
\begin{align*}
b= c \beta + \sum_{x \in \Ps^1_k(\overline{k})} \sum_{(i,p)=1} c_{x,i} [\pi_x]^{-i} \in W(K) \subset W(\overline{k}(X))
\end{align*}
with $c \in W(\F_p)$, $c_{x,i} \in W(\overline{k})$ such that
\begin{enumerate}
\item
for $\sigma \in \Gal(\overline{k}/k)$, $x \in \Ps^1_k(\overline{k})$: $c_{\sigma x,i}=\sigma c_{x,i}$
\item
for every integer $n \in \Z_{\geq 1}$,  there are only finitely many $c_{x,i}$ with $v(c_{x,i}) \leq n$. 
\end{enumerate}
\end{lemma}

We can then study $\Z_p$-towers of $K$ given by a reduced vector $a$. Note that reduced vectors $a, a' \in W(K)/\wp W(K)$ induce the same tower if and only $\langle a \rangle = \langle a' \rangle$ if and only if $a=ca'$ with $c \in W(\F_p)^*=\Z_p^*$. One can easily read this off from the expressions of $a$ and $a'$. 

\begin{proposition} \label{hulu2}
Let $a= c \beta + \sum_{x \in \Ps^1_k(\overline{k})} \sum_{(i,p)=1} c_{x,i} [\pi_x]^{-i} \in W(K)$ be given in the unique form as in Lemma \ref{hulu}.  Assume 
\begin{align*}
\mathrm{min}( \{v(c_{x,i}): x \in \Ps^1_k(\overline{k}), (i,p)=1\} \cup \{v(c)\})=0.
\end{align*}
Consider the tower $K_{\infty}/K$ given with field $K_n=K(\wp^{-1} \pi_n(a))$ of genus $g_n$. For $x \in \Ps^{1}_k(\overline{k})$,  set 
\begin{align*}
u_{x,n} = \left\{  \begin{array}{cc} 1+ \max\{ ip^{n-v(c_{x,i})-1}  : (i,p)=1 \textrm{ s.t. } v(c_{x,i})<n\} & \mathrm{if\ }\exists i: v(c_{x,i})<n \\	
0 & \mathrm{otherwise}. \end{array} \right.
\end{align*}
One has $[K_n:K]=p^n$, 
\begin{align*}
n_u=n_c=\mathrm{min}( v(c_{x,i}): x \in \Ps^1_k(\overline{k}), (i,p)=1)
\end{align*}
and 
\begin{align*}
p^{\min(n_c,n)}(2 g_n - 2) = -2p^{n} + \sum_{x \in \Ps^1_k(\overline{k})} \sum_{j=0}^n \varphi(p^j) u_{x,j}.
\end{align*}
Furthermore, a prime $\pa$ is ramified in $K_n/K$ if and only if for a representative $x \in \Ps^1_k(\overline{k})$ of $\pa$ one has $u_{x,n}>0$. 
\end{proposition}
\begin{proof}
The first assumption on $a$ assures that $[K_n:K]=p^n$ for all $n$. We will now compute $g_n$. Let $k'=k(x: \exists i, v(c_{x,i})<n)$, a finite field extension of $k$. We define $K_i'=K_i k'$. One has $g_i=g(K_i)=g(k' K_i)$, in particular $g_n=g(K_n')$. Consider the extension $K_n'/K_0'$, which ramifies precisely at the $x$ such that there is an $i$ with $v(c_{x,i})<n$, that is, all the ramifying primes became rational and for such a prime $x$ one has $[k(x)k':k']=1$ (Proposition \ref{712}). One can apply Theorem \ref{831}, Proposition \ref{715} and Proposition \ref{712},  to see $u_{x,n}$ corresponds to the exponent in the conductor at $x$ in the extension $K_n'/K_0'$ and one has
\begin{align*}
p^{\min\{n_c,n \}}(2 g_n - 2)=& p^{n} ( 2 g_0 - 2 )+ \sum_{x} [k(x)k':k'] \sum_{j=0}^n \varphi(p^j) u_{x,j}.
\end{align*} 
Since $g_0=0$, the result follows.
\end{proof}

\begin{example}
Consider the unit root $\Z_p$-extension (called the Artin-Schreier-Witt extension in \cite{WAN7} ) given by the unit root coefficient 
polynomial 
\begin{align*}
x=\sum_{(i,p)=1}^d [b_i X^{i}]=\sum_{(i,p)=1}^d [b_i][X^{i}] \in W(K)
\end{align*}
with $b_i \in k$ and $b_d \neq 0$, $d>0$ not divisible by $p$.  By the above equation, this defines a $\Z_p$-extension which is totally ramified at $\infty$. One finds for $n \geq 1$:
\begin{align*}
u_{\infty,n} = 1+ d p^{n-1}.
\end{align*}
and
\begin{align*}
2g_n-2 =& -2p^n + \sum_{j=1}^n \varphi(p^j) ( 1+ d p^{j-1} )  = -2p^n + p^n-1  + d \frac{p^{2n}-1}{p+1}  \\
=& \frac{d}{p+1}p^{2n} - p^n - \frac{p+1+d}{p+1}.
\end{align*}
This is an example of a genus-stable tower. The stable arithmetic properties of these extensions have been studied and established in \cite{WAN7}. 

\end{example}

\begin{remark}
Let $a=(a_0,a_1, \ldots) \in W(K)$ with $a_0 \not \in \wp W(K)$. Consider the corresponding $\Z_p$-extension given by $a$. Let $\pa$ be a prime of  $K(X)$ of degree $d'$ over $\F_p$ which does not ramify in the tower. We give a geometric way to compute the Frobenius element $(\pa, K_{\infty}/K)$. Let $z \in \Ps^1(\overline{k})$ be a representative of $\pa$. Assume that $z$ is not a pole of the $a_i$ (otherwise, we have to find another representative of $a \pmod{\wp W(K)}$; or one can assume $a$ is in our unique reduced form). Set $a(z)=(a_0(z),a_1(z),\ldots) \in W(k(z))$. Let $y \in \wp^{-1} z \in W(\overline{k})$.  
One has
\begin{align*}
F^{d'}y = y + \sum_{j=0}^{d'-1} F^j(Fy-y) = y+ \mathrm{Tr}_{W(k(z))/W(\F_p)} ( a(z) ).
\end{align*}
This shows that the Frobenius is equal to 
\begin{align*}
(\pa, K_{\infty}/K)= \left(a \mapsto -\mathrm{Tr}_{W(k(z))/W(\F_p)} (a(z)) \right) \subseteq \Hom(\langle a \rangle, W(\F_p))  \cong \Gal(K(\wp^{-1}a)/K).
\end{align*} 
A similar formula works for primes which are not ramified in say $K_n/K$. See \cite[Theorem 3.3.7]{ST} and \cite[Section 3.2.4]{LIU} for more details on this geometric approach. Furthermore, this formula generalizes when $K$ is replaced by another function field.
\end{remark}

\begin{remark}
Some related work on global function fields has been done in the PhD thesis of Shabat \cite{SHAB}. In his thesis, one uses Artin-Schreier-Witt extensions to construct curves with a lot of points in comparison to its genus. Explicit formulae for computing the genus and number of points of the composite of  Artin-Schreier-Witt extensions can be found there.
\end{remark}

\section{Thanks}
We would like to thank Chris Davis for his help with Witt-vectors and for his proofreading of parts of this manuscript.

\end{document}